\documentclass[11pt]{amsart}
\usepackage{amsmath,amssymb,amsthm,upref,graphicx,mathrsfs, enumerate,esint}
\usepackage{esint}
\usepackage{textcomp}
\usepackage{array}
\usepackage{color,xcolor}
\usepackage{latexsym}
\usepackage{amsmath}
\usepackage{amssymb}
\usepackage{amsthm} 
\usepackage{amscd}
\usepackage{color}
\usepackage{comment}

\numberwithin{equation}{section}
\allowdisplaybreaks

\newtheorem{theorem}{Theorem}[section]
\newtheorem{proposition}[theorem]{Proposition}
\newtheorem{lemma}[theorem]{Lemma}
\newtheorem{corollary}[theorem]{Corollary}

\theoremstyle{definition}
\newtheorem{definition}[theorem]{Definition}
\newtheorem{example}[theorem]{Example}
\newtheorem{remark}[theorem]{Remark}

\begin{document}
\title[Ridgelet transforms in Banach lattices]{Ridgelet Transforms of Functions in Banach lattices}

\author{Mitsuo Izuki}
\thanks{M.\,Izuki\,:\,Faculty of Liberal Arts and Sciences, Tokyo City University}

\author{Takahiro Noi}
\thanks{T.\,Noi\,:\,Department of Mathematical and Data Science, Otemon Gakuin University}

\author{Yoshihiro Sawano}
\thanks{Y.\,Sawano\,:\,Department of Mathematics, Graduate School of Science and Engineering, Chuo University}

\author{Hirokazu Tanaka}
\thanks{H.\,Tanaka\,:\,Faculty of Information Technology, Tokyo City University}

\date{}

\address{
Faculty of Liberal Arts and Sciences\\
Tokyo City University\\
1-28-1 Tamadutsumi, Setagaya-ku, Tokyo 158-8557, Japan
}
\email{izuki@tcu.ac.jp}

\address{
Department of Mathematical and Data Science\\
Otemon Gakuin University\\
2-1-15 Nishiai, Ibaraki, Osaka 567-8502, Japan
}
\email{taka.noi.hiro@gmail.com}

\address{
Department of Mathematics\\
Graduate School of Science and Engineering, Chuo University\\
1-13-27 Kasuga, Bunkyo-ku, Tokyo 112-8551, Japan
}
\email{yoshihiro-sawano@celery.ocn.ne.jp}

\address{
Faculty of Information Technology\\
Tokyo City University\\
1-28-1 Tamadutsumi, Setagaya-ku, Tokyo 158-8557, Japan
}
\email{htanaka@tcu.ac.jp}

\maketitle
{\bf Key words{\rm:}}
ridgelet transform $\cdot$
$k$-plane Radon transform $\cdot$
Banach lattices $\cdot$
fractional integrals $\cdot$
reproducing formulas

{\bf Mathematics Subject Classification (2020){\rm:}}
42B35 $\cdot$ 46E30 $\cdot$ 42B20 $\cdot$ 42B25
\begin{abstract}
We establish a reproducing formula for the ridgelet transform on $\mathbb{R}^n$ in the framework of Banach lattices introduced in a recent paper by Nieraeth.
Our approach is based on the $k$-plane Radon transform and a wavelet-type
reconstruction operator acting on functions defined on the Grassmannian of
$k$-dimensional affine planes.
Under mild structural assumptions on the underlying Banach lattice, 
we prove that the ridgelet reconstruction converges both in the
lattice norm and almost everywhere.
The admissibility conditions on the wavelet function are formulated in terms
of the Riemann--Liouville fractional integral.  
As a consequence, we obtain explicit inversion formulas for functions in
a Banach lattice $X$ which is contained in
$L^1({\mathbb R}^n)+L^p(\mathbb{R}^n)$
with some constant $1 \le p < \frac{n}{k}$, together with precise expressions for the
reconstruction constant.
These results provide a unified framework for ridgelet-type reproducing formulas
in a broad class of function spaces beyond the classical $L^p$ setting.
\end{abstract}

\section{Introduction}
\label{s1}

The aim of this paper is to establish a reproducing formula for the ridgelet
transform on $\mathbb{R}^n$ and the
$k$-plane Radon transform
in the setting of Banach lattices~$X$.
By a Banach lattice we mean a Banach space of (equivalence classes of)
measurable functions on $\mathbb{R}^n$ equipped with a lattice structure
such that, whenever $f \in X$ and $g$ is a measurable function satisfying
$|g(x)| \le |f(x)|$ almost everywhere, one has $g \in X$ and $\|g\|_X \le \|f\|_X$.
We will assume inclusion (into a certain sum space of Lebesgue spaces), density of $C_{\rm c}({\mathbb R}^n)$ and
stability under averaging. See Theorem \ref{thm:wavelet-reconstruction}.
Typical examples include Lebesgue spaces, Lorentz spaces, and Orlicz spaces.
Standard references on Banach lattices include
\cite{BennettSharpley1988, KreinPetuninSemenov1982, RaoRen1991}.

To formulate our results, we will recall another example of Banach lattices.
For $1 \le p < \infty$, we denote by $L^1({\mathbb R}^n) + L^p({\mathbb R}^n)$ the sum space consisting of all
measurable functions $f$ on $\mathbb{R}^n$ that admit a decomposition
\[
f = f_1 + f_p,
\qquad
f_1 \in L^1(\mathbb{R}^n), \quad f_p \in L^p(\mathbb{R}^n).
\]
This space is equipped with the norm
\[
\|f\|_{L^1+L^p}
=
\inf\bigl\{
\|f_1\|_{L^1} + \|f_p\|_{L^p}
\,:\, f = f_1 + f_p,
f_1 \in L^1(\mathbb{R}^n), 
f_p \in L^p(\mathbb{R}^n)
\bigr\}.
\]
Thus $L^1({\mathbb R}^n) + L^p({\mathbb R}^n)$ is a Banach lattice.

Next, we briefly recall
the definition of the
$k$-plane Radon transform.
Let $k \in \mathbb{N}$ with $1 \le k < n$.
Denote by ${\mathcal G}_{n,k}$ the Grassmannian of $k$-dimensional affine planes
in $\mathbb{R}^n$, equipped with its canonical measure
$d\mu_{{\mathcal G}_{n,k}}$.
For a plane $\tau \in {\mathcal G}_{n,k}$, we write $d_\tau x$ for the Lebesgue
measure on~$\tau$.
We recall the precise definition of  ${\mathcal G}_{n,k}$ and its canonical measure
more precisely in Section \ref{s2}.

\begin{definition}[$k$-plane Radon transform]\label{def:k-plane-radon}
Let $f$ be a locally integrable function on $\mathbb{R}^n$.
The $k$-plane Radon transform of $f$ is defined by
\begin{equation}\label{eq:k-plane-radon}
\hat f(\tau)
=
\int_{\tau} f(x)\, d_\tau x,
\qquad \tau \in {\mathcal G}_{n,k},
\end{equation}
provided that the integral exists.
\end{definition}

The mapping $f \mapsto \hat f$ is well defined only for functions
$f \in L^p(\mathbb{R}^n)$ with $1 \le p < \frac{n}{k}$.
A sharp restriction on the admissible range of~$p$ was established by
Solmon~\cite{Solmon1979}.
If $f \in L^p(\mathbb{R}^n)$, then $\hat f(\tau)$ is finite for almost every
$\tau \in {\mathcal G}_{n,k}$ provided that
$1 \le p < \frac{n}{k}$.
On the other hand, if $p \ge \frac{n}{k}$ and
\[
f(x) = (2+|x|)^{-n/p}\bigl(\log(2+|x|)\bigr)^{-1},
\]
then $f \in L^p(\mathbb{R}^n)$ while $\hat f(\tau) := \infty$
for all $\tau \in {\mathcal G}_{n,k}$.
Therefore, throughout this paper we consider function spaces~$X$ that are
continuously embedded into $L^1({\mathbb R}^n) + L^p({\mathbb R}^n)$ for some
$p \in [1, \frac{n}{k})$.

Increasing interest in the transform~\eqref{eq:k-plane-radon}
has been motivated by developments in approximation theory and computer science,
notably through its connections with ridgelet-type representations.
This line of research was initiated and further developed by
D.~Donoho, E.~Cand\`es, and their collaborators; see, for example,
\cite{Candes1998, Donoho1993}.
Remark that Murata and Rubin independently introduced
wavelet-like transform \cite{Murata1996}.
See \cite{SonodaMurata2017} for 
neural network with unbounded activation functions.
We next introduce a wavelet-type reconstruction operator.
Let $w$ be a wavelet function on $[0,\infty)$ satisfying suitable
regularity and decay assumptions.
In particular, we assume that
\begin{equation}\label{eq:260118-1}
\int_{0}^{R} r^{n-k-1}\,|w(r)|\,dr < \infty
\end{equation}
for every $R>0$.
Further conditions on $w$ will be specified below.

For $t>0$ and a measurable function
\[
\varphi \colon \mathcal G_{n,k} \to \mathbb C,
\]
we define the  smoothened dual $k$-plane transform $W_t^{*}$ by
\begin{equation}\label{eq:W-star-def}
W_t^{*}\varphi(x)
=
t^{-n}
\int_{\mathcal G_{n,k}}
\varphi(\tau)\,
w\!\left(\frac{|x-\tau|}{t}\right)
\,d\mu_{\mathcal G_{n,k}}(\tau),
\qquad x \in \mathbb R^{n},
\end{equation}
whenever the integral on the right-hand side is well defined.
Here $|x-\tau|$ denotes the Euclidean distance
between the point $x$ and the $k$-plane $\tau$, and
$\mu_{\mathcal G_{n,k}}$ denotes the canonical measure on
the Grassmannian $\mathcal G_{n,k}$.

We denote by
\[
\sigma_{n-1} := \frac{2\pi^{\frac{n}{2}}}{\Gamma(\frac{n}{2})}
\]
the surface area of the unit sphere $S^{n-1} \subset \mathbb R^n$.

\begin{definition}\label{def:admissible-pair}
A radial function $w$ on $\mathbb R^{n-k}$ is called
\emph{admissible} if it satisfies
\eqref{eq:260118-1} and
the decay condition
\begin{equation}\label{eq:radial-decay-condition}
\int_{\{y \in \mathbb R^{n-k} \,:\, |y|>1\}}
|y|^{\beta}\,|w(y)|\,dy
=
\sigma_{n-k-1}\int_1^\infty r^{\beta+n-k-1}|w(r)|\,dr
< \infty,
\end{equation}
for some constant $\beta>k$
as well as  the cancellation conditions
\begin{equation}\label{eq:radial-moment-conditions}
\int_{\mathbb R^{n-k}}
|y|^{j}\, w(y)\,dy
=
\sigma_{n-k-1}\int_0^\infty r^{j+n-k-1}w(r)\,dr
= 0,
\qquad
j = 0,2,4,\dots,2\left[\frac{k}{2}\right].
\end{equation}
\end{definition}

This terminology is a slight abuse in comparison with \cite{Rubin2004}.
Indeed, in \cite{Rubin2004} a pair $(u,v)$ is said to be admissible
if the function
\begin{equation}\label{eq:def-w-convolution}
w(y)
=
u*v(y)
=
\int_{\mathbb R^{n-k}}
u(z)\,v(y-z)\,dz,
\qquad y \in \mathbb R^{n-k},
\end{equation}
satisfies the conditions stated in Definition~\ref{def:admissible-pair}.

\begin{example}
Let
\[
u(y)=\Delta^{N_1}e^{-|y|^2},
\qquad
v(y)=\Delta^{N_2}e^{-|y|^2},
\qquad y \in \mathbb R^{n-k},
\]
where $N_1$ and $N_2$ are nonnegative integers such that
\[
\left[\frac{k}{2}\right] -1
<
N_1+N_2
\le
\left[\frac{k}{2}\right].
\]
Then the function $w=u*v$ is admissible.
Further examples of admissible wavelet functions can be found in
\cite[Example~5.4]{Rubin2002}.
\end{example}
We next turn to the definition of Banach lattices.
In order to formulate our results, we adopt the terminology introduced by
Nieraeth~\cite{Nieraeth26}.
For the reader's convenience, we recall the definition of Banach function
spaces on $\mathbb{R}^n$ in Definition~\ref{def:BFS}.
Throughout this paper, for a measurable set $E \subset \mathbb{R}^n$, we denote by
$\chi_E$ its indicator function.

\begin{definition}[Muckenhoupt condition]\label{def:Muckenhoupt}
Let $X$ be a Banach function space on $\mathbb{R}^n$.
We say that $X$ satisfies the \emph{Muckenhoupt condition} if
$\chi_Q \in X$ and $\chi_Q \in X'$ for every cube
$Q \subset \mathbb{R}^n$, and there exists a constant $C \ge 1$ such that
\[
\|\chi_Q\|_{X}\,\|\chi_Q\|_{X'}
\le C\,|Q|
\quad
\text{for all cubes } Q \subset \mathbb{R}^n .
\]
In this case, we write $X \in {\mathcal A}$ and denote by $[X]_{\mathcal A}$ the smallest constant
$C$ for which the above inequality holds.
\end{definition}

Denote by $B(x,R)$ the open ball in ${\mathbb R}^n$ centered at $x$ with radius $R>0$.
Vinogradov employed a different notation in \cite{Vinogradov2024}.
In his terminology, spaces satisfying
Definition~\ref{def:Muckenhoupt} are referred to as
\emph{spaces with bounded averaging}.
From this viewpoint, the class $\mathcal A$ of Banach lattices can be interpreted as
expressing a form of \emph{stability under averaging}:
there exists a constant $C>0$ such that for every $r>0$ and every
$f \in X$,
\[
\left\|
\fint_{B(\cdot,r)} f(y)\,dy
\right\|_{X}
\le C\,\|f\|_{X}.
\]
Here $\fint$ stands for the average.

The following theorem,
extending \cite[Theorem 3.1]{Rubin2004}, serves as a starting point of this paper.

\begin{theorem}\label{thm:wavelet-reconstruction}
Let $X \in {\mathcal A}$ be a Banach lattice satisfying the following conditions{\rm:}
\begin{enumerate}
\item[$(1)$]
Inclusion{\rm:}
$X \subset L^1({\mathbb R}^n) + L^p({\mathbb R}^n)$ for some
$1 \le p < \frac{n}{k}$.
\item[$(2)$]
Density{\rm:}
$C_{\mathrm c}(\mathbb{R}^n)$ is dense in~$X$.
\end{enumerate}
Assume that $\psi$ is a measurable function on $\mathbb{R}_+$, and that
there exists an admissible function $w$ such that
\begin{equation}\label{eq:psi-w-relation}
\psi(r)
=
c_{k,n}\,r^{2-n}
\int_0^{r}
s^{n-k-1}
w(s)\,
(r^2-s^2)^{\frac{k}{2}-1}
\,ds,
\qquad
c_{k,n}=\frac{\sigma_{k-1}\sigma_{n-k-1}}{\sigma_{n-1}},
\end{equation}
Assume that $\psi$ admits an integrable decreasing majorant.
Define the smoothened dual $k$-plane transform $W_t^*$ by \eqref{eq:W-star-def}.
Then
the inversion formula
\begin{equation}\label{eq:dt-over-t-reconstruction}
\lim_{t \downarrow 0}
W_t^{*}\hat f(x)
=
c\,f(x)
\end{equation}
holds,
where
\[
c
=
\int_{{\mathbb R}^n} \psi(|x|)\,dx.
\]
The convergence holds in the topology of $X$ and almost everywhere.
\end{theorem}

A typical class of examples
in ${\mathcal A}$ is the one of Banach function spaces on which
the Hardy--Littlewood maximal operator is bounded.
Recall that the \emph{Hardy--Littlewood maximal operator} $M$ is defined by
\[
Mf(x)
:=
\sup_{\substack{\text{cubes } Q \ni x}}
\fint_{Q} |f(y)|\,dy,
\qquad x \in \mathbb{R}^n,
\]
for any measurable function $f$ on $\mathbb{R}^n$.

As shown in \cite{Nieraeth26,Vinogradov2024}, if $M$ is bounded on $X$, then
$X \in \mathcal A$ and
\begin{equation}\label{eq:260206-22}
[X]_{\mathcal A} \le c_0 \|M\|_{X \to X},
\end{equation}
where $c_0>0$ is an absolute constant.
Moreover, a simple geometric argument yields
\begin{equation}\label{eq:260206-21}
\left\|
\fint_{B(\cdot,r)} f(y)\,dy
\right\|_{X}
\le
c_n [X]_{\mathcal A}\,\|f\|_{X},
\qquad f \in X,
\end{equation}
where $B(x,r)$ denotes the Euclidean ball of radius $r>0$ centered at $x$,
and $c_n>0$ depends only on the dimension $n$.

Additional examples beyond the Banach lattices discussed above are presented in Sections~\ref{s4} and~\ref{s5}. 
We note, however, that the converse implication may fail, as demonstrated by the example 
\( X = L^1(\mathbb{R}^n) \).

Another counterexample can be obtained as follows: 
If the converse were true, then the maximal operator \( M \) would be bounded on both \( X \) and its associate space \( X' \). 
However, this contradicts an example involving a weighted Morrey space due to Tanaka~\cite[Proposition~4.2]{Tanaka2015}.

The density assumption in Theorem~\ref{thm:wavelet-reconstruction} is of a technical nature. 
Suppose that \( X \) does not satisfy this condition. 
In that case, we denote by \( \widetilde{X} \) the completion of \( C^\infty_{\mathrm c}(\mathbb{R}^n) \) in the Banach space \( X \). 
Observe that convolution operators with radially decreasing kernels preserve \( \widetilde{X} \). 
Therefore, Theorem~\ref{thm:wavelet-reconstruction} can be applied to \( \widetilde{X} \) instead.
Theorem~\ref{thm:wavelet-reconstruction} constitutes the core of the
convolution--backprojection method.

Instead of passing to the limit in
Theorem~\ref{thm:wavelet-reconstruction} as $t\downarrow 0$,
one may integrate with respect to the Haar measure
$dt/t$ over $(0,\infty)$.
 This formalism leads to the following:
\begin{theorem}\label{thm:wavelet-reconstruction1}
Let $X$ be a function space
as in Theorem \ref{thm:wavelet-reconstruction}.
Assume that $\psi$ is a measurable function on $\mathbb{R}_+$, and that
there exists an admissible function $w$ such that
\eqref{eq:psi-w-relation} is satisfied.
We write
\begin{align}
\label{eq:260128-1}
\tilde w(s)
&=s^{\frac{n-k}{2}-1}w(\sqrt{s}),\\
\label{eq:260129-1}
\lambda(s)
&=\frac{1}{s\Gamma(\frac12k+1)}\int_0^{s^2}(s^2-r)^{\frac{k}{2}}\tilde w(r)\,dr.
\end{align}
Set
\begin{equation}\label{eq:260119-26}
\tilde\psi(x)
=
\int_{1}^{\infty}
\psi\!\left(\frac{|x|}{t}\right)
\,\frac{dt}{t^{1+n}} \quad (x \in {\mathbb R}^n).
\end{equation}
Assume that $\tilde\psi$ admits an integrable decreasing majorant.
Define the operator $W_t^*$ by \eqref{eq:W-star-def}.
Then
the inversion formula
\begin{equation}\label{eq:dt-over-t-reconstruction}
\int_0^{\infty}
W_t^{*}\hat f(x)\,\frac{dt}{t}
=
\lim_{\varepsilon \downarrow 0}
\int_{\varepsilon}^{\infty}
W_t^{*}\hat f(x)\,\frac{dt}{t}
=
c\,f(x)
\end{equation}
holds,
where
\begin{equation}\label{eq:260206-11}
c
=
\frac{\pi^{\frac{n}{2}}}{\Gamma\left(\frac{n-k}{2}\right)} \int_0^\infty \lambda(s)ds.
\end{equation}
The convergence holds in the topology of $X$ and almost everywhere.
\end{theorem}
Theorem \ref{thm:wavelet-reconstruction1} extends
\cite[Theorem 3.2]{Rubin2004}.

The function $\psi$, which realizes approximation to the identity,
can be chosen arbitrarily within the assumptions in
Theorem \ref{thm:wavelet-reconstruction1}.
We also recall the definition of the Riemann--Liouville fractional integral.
For a complex number $\alpha$ satisfying
$\Re \alpha > 0$ and a suitable function $g$, it is given by
\begin{equation}\label{eq:RL-fractional-integral}
I_+^\alpha g(t)
=
\frac{1}{\Gamma(\alpha)}
\int_0^t g(r)\,(t-r)^{\alpha-1}\,dr,
\qquad t>0.
\end{equation}
The generating function $w$ is then uniquely determined by $\psi$
as a solution of the Abel-type integral equation
\eqref{eq:psi-w-relation}.

 For later considerations, we recall an identity
  assuming that $\alpha$ in \cite[Lemma~2.4]{Rubin1998} is a real number:
\begin{lemma}{\rm \cite[Lemma~2.4]{Rubin1998}}\label{lem:260307-1}
Let $\beta>\alpha>0$, and assume
\begin{align*}
\int_0^\infty t^j \psi(t)\, dt &= 0 \quad \text{for all } j = 0,1,\ldots, [\alpha], \\
\int_0^\infty t^\beta |\psi(t)|\, dt &<\infty.
\end{align*}
Define
\[
\gamma = \alpha - \min([1+\alpha], \beta) < 0.
\]
Then
\[
I_+^{1+\alpha} \psi(s) =
\begin{cases}
O(s^\alpha), & 0 < s \le 1,\\
O(s^\gamma), & s \ge 1.
\end{cases}
\]
 Furthermore,
\[
\int_0^\infty I_+^{1+\alpha} \psi(s)\, ds =
\begin{cases}
\Gamma(-\alpha) \displaystyle \int_0^\infty s^\alpha \psi(s)\, ds, & \alpha \notin \mathbb{N},\\[1ex]
\displaystyle \frac{(-1)^{\alpha+1}}{\alpha!} \int_0^\infty s^\alpha \psi(s) \log s\, ds, & \alpha \in \mathbb{N}.
\end{cases}
\]
\end{lemma}

In the course of the proof of
Theorem \ref{thm:wavelet-reconstruction1}
we will justify that the composition
$W_t^{*}\hat f(x)$ makes sense for almost all
$x \in {\mathbb R}^n$
as long as $f \in X$.

If $w$ satisfies
\begin{equation}\label{eq:moment-conditions}
\int_{0}^{\infty}
s^{j+n-k-1}\, w(s)\,ds = 0,
\qquad
j = 0,2,4,\dots,2\left[\frac{k}{2}\right],
\end{equation}
 and 
\begin{equation}\label{eq:decay-condition}
\int_{1}^{\infty}
s^{\beta+n-k-1}\,|w(s)|\,ds < \infty,
\end{equation} then, by the work of
Rubin, the constant $c$ appearing in 
\eqref{eq:dt-over-t-reconstruction} and given by 
\eqref{eq:260206-11} admits the following explicit representation:
\begin{equation}\label{eq:constant-c-explicit}
c
=
\frac{\pi^{\frac{n}{2}}}{\Gamma\left(\frac{n-k}{2}\right)}  
\times 
\begin{cases}
\displaystyle
\Gamma\left(-\frac{k}{2}\right)
\int_{0}^{\infty} s^{n-1}\, w(s)\,ds,
& \text{if $k$ is odd}, \\[3ex]
\displaystyle
\frac{2(-1)^{1+\frac{k}{2}}}{ \left( \frac{k}{2} \right)!}
\int_{0}^{\infty} s^{n-1}\, w(s)\,\log s\,ds,
& \text{if $k$ is even},
\end{cases}
\end{equation} 
See \cite[Lemma~2.4]{Rubin1998}.
See also the paper by Saeki for an account of why $\log$ appears here
\cite{Saeki1995}.
Consequently, we obtain the following corollary:

\begin{corollary}\label{cor:inversion-constant}
Suppose that $X$ is a Banach lattice  as in Theorem \ref{thm:wavelet-reconstruction}.
Let $f \in X$.
If the function $w$ satisfies \eqref{eq:moment-conditions} and
\eqref{eq:decay-condition}, then the inversion formula
\eqref{eq:dt-over-t-reconstruction} holds with the constant $c$ given by
\eqref{eq:260206-11} and \eqref{eq:constant-c-explicit}.
\end{corollary}

One may interpret $W_t^{*}\varphi$ as the wavelet transform of $\varphi$
associated with the dual $k$-plane transform and generated by the wavelet function $w$.
Usually, the operator $f \mapsto W_t^*\hat{f}$ is decomposed into two operators{\rm:}
\[
W_t^* \hat{f} = V_t^* U_t f.
\]
Here $V_t^*$ and $U_t$ are described 
in Definition \ref{def:k-plane-ridgelet}
below{\rm:}
\begin{definition}[$k$-plane ridgelet-type transforms]\label{def:k-plane-ridgelet}
Fix an integer $k$ with $1 \le k < n$.
Let $u$
and $v$ be suitable (wavelet) functions on $[0,\infty)$.
For $x \in \mathbb{R}^n$ and $\tau \in {\mathcal G}_{n,k}$,
we write $|x-\tau|$ for the Euclidean distance from $x$ to $\tau$.
\begin{enumerate}
\item
Let $t>0$, and let $f \colon \mathbb{R}^n \to \mathbb{C}$ be a measurable function.
Define
the generalized projection operator
(smoothed 
$k$-plane Radon transform) by
\begin{equation}\label{eq:Ut}
U_t f(\tau)
=
t^{k-n}
\int_{\mathbb{R}^n}
f(x)\,
u\!\left(\frac{|x-\tau|}{t}\right)
\,dx,
\qquad \tau \in {\mathcal G}_{n,k}.
\end{equation}

\item
Let $t>0$, and let $\varphi \colon {\mathcal G}_{n,k} \to \mathbb{C}$ be a measurable function.
Define
the generalized backprojection operator
(smoothed dual 
$k$-plane operator) by
\begin{equation}\label{eq:Vt-star}
V_t^{*}\varphi(x)
=
t^{k-n}
\int_{{\mathcal G}_{n,k}}
\varphi(\tau)\,
v\!\left(\frac{|x-\tau|}{t}\right)
\,d\mu_{\mathcal{G}_{n,k}}(\tau),
\qquad x \in \mathbb{R}^n.
\end{equation}
\end{enumerate}
In \eqref{eq:Ut} and \eqref{eq:Vt-star},
we suppose that the integral on the right-hand side converges.
\end{definition}

In the limiting case $k=0$, the manifold ${\mathcal G}_{n,0}$ can be identified with
$\mathbb{R}^{n}$, and the transforms~\eqref{eq:Ut} and~\eqref{eq:Vt-star}
reduce to the classical continuous wavelet transforms on $\mathbb{R}^{n}$;
see, for example, \cite{FJW1991}.
As the following theorem shows, the function $f$ can be reconstructed
through an integral decomposition involving the operators $V_t^{*}$ and
$U_t$.

\begin{theorem}\label{Theorem 3.5}
Let $X$ be a Banach lattice satisfying the assumptions of
Theorem~\ref{thm:wavelet-reconstruction}.
Let $u, v : [0,\infty) \to \mathbb{R}$ be measurable functions such that
\[
\int_0^\infty s^{\,n-k-1} \bigl( |u(s)| + |v(s)| \bigr)\, ds < \infty .
\]
Identify $u$ and $v$ with radial functions on $\mathbb{R}^{\,n-k}$, so
that both are integrable over $\mathbb{R}^{\,n-k}$, and define
\[
w := u * v .
\]

Assume that $w$, viewed as a function on $(0,\infty)$, is admissible,
that is, it satisfies the moment conditions~\eqref{eq:moment-conditions}
and the decay condition~\eqref{eq:decay-condition}.
Let $c$ be the constant defined in~\eqref{eq:constant-c-explicit}.
Suppose furthermore that the function $\tilde{\psi}$, defined
in~\eqref{eq:260119-26}, admits an integrable decreasing majorant.

Then, for every $f \in X$,
\[
\int_0^\infty V_t^{*} U_t f(x)\, \frac{dt}{t^{\,1+k}}
=
\lim_{\varepsilon \to 0^+}
\int_\varepsilon^\infty V_t^{*} U_t f(x)\, \frac{dt}{t^{\,1+k}}
=
c\, f(x),
\]
with convergence both in $X$ and almost everywhere.
\end{theorem}
Here and throughout this paper we use the following convention.
Let $A, B \ge 0$.
We write $A \lesssim B$ if there exists a constant $C>0$ such that
\[
A \le C B,
\]
where the constant $C$ is usually independent of the functions under
consideration.
The notation $A \sim B$ means that both inequalities
$A \lesssim B$ and $B \lesssim A$ hold.

The remainder of this paper is organized as follows:
Section~\ref{s2} recalls some preliminary facts.
Section~\ref{s3} is devoted to the proofs of the theorems stated in the
Introduction.
In Section~\ref{s4} we present examples of Banach lattices.
Finally, Section~\ref{s5} compares the Muckenhoupt condition
of Banach lattices with
translation invariance.

\section{Preliminaries}
\label{s2}

We recall the preliminary material needed in the sequel.
Following \cite{Nieraeth26,Nogayama26}, we review the notion of Banach
lattices in Section~\ref{s2.0}.
Sections~\ref{s2.1}--\ref{s2.2} are based on
\cite[Chapter~12]{Rubin2015}.
Section~\ref{s2.1} is devoted to a compact manifold, namely the linear
Grassmann manifold, which serves as a basic ingredient for
Section~\ref{s2.2a}.
Section~\ref{s2.2a} constitutes the main technical tool of this paper,
where we study the affine Grassmann manifold.
Section~\ref{s2.2} recalls the $k$-plane Radon transform and the ridgelet
transform.
Finally, Section~\ref{s2.5} is oriented in a different direction, where
we establish an averaging lemma inspired by
\cite{Vinogradov2024}.
\subsection{Banach lattices}
\label{s2.0}

We consider the following conditions for a Banach function space $X$.

\begin{itemize}
\item[(L)] \textbf{(Lattice property)}
For all $f \in X$ and $g \in L^{0}(\mathbb{R}^{n})$, if
$|g| \le |f|$ almost everywhere, then $g \in X$ and
\[
\|g\|_{X} \le \|f\|_{X}.
\]

\item[(F)] \textbf{(Fatou property)}
If a norm bounded sequence $\{f_j\}_{j\in\mathbb{N}} \subset X$ satisfies
$0 \le f_j \uparrow f$ almost everywhere,
then 
\[
\|f\|_{X} = \sup_{j\in\mathbb{N}} \|f_j\|_{X}.
\]

\item[(Si)] \textbf{(Simple integrability)}
For any measurable set $E \subset \mathbb{R}^{n}$ with finite measure,
$\chi_E \in X$.

\item[(BSi)] \textbf{(Ball simple integrability)}
For any ball $B \subset \mathbb{R}^{n}$,
$\chi_B \in X$.

\item[(LI)] \textbf{(Local ideal property)}
For any measurable set $E \subset \mathbb{R}^{n}$ with finite measure and
any $f \in X$, there exists a constant $C>0$, independent of $f$, such that
\[
\|f\|_{L^1(E)} \le C \|f\|_{X}.
\]

\item[(BLI)] \textbf{(Ball local boundedness)}
For any ball $B \subset \mathbb{R}^{n}$ and any $f \in X$, there exists a constant
$C>0$, independent of $f$, such that
\[
\|f\|_{L^1(B)} \le C \|f\|_{X}.
\]

\item[(Sa)] \textbf{(Saturation property)}
For every measurable subset $E \subset \mathbb{R}^{n}$ with positive measure,
there exists a measurable set $F \subset E$ with nonzero measure such that
\[
\chi_F \in X.
\]
\end{itemize}

Based on this, we provide the following definitions
\cite[Definition 1.1]{Nogayama26}:
\begin{definition}[Banach function spaces]\label{def:BFS}\
\begin{enumerate}
\item
A Banach space $X$ is called a \emph{Banach function space} if $X$ enjoys the
properties {$\rm(L)$}, $\rm(F)$, $\rm(Si)$ and $\rm(LI)$.

\item
A Banach space $X$ is called a \emph{ball Banach function space} if $X$ enjoys the
properties
$\rm(L)$, $\rm(F)$, $\rm(BSi)$ and
$\rm(BLI)$.

\item
A Banach space $X$ is called a \emph{saturated Banach function space} if $X$
enjoys the properties $\rm(L)$, $\rm(F)$ and $\rm(Sa)$.
\end{enumerate}
\end{definition}

\subsection{The Linear Grassmannian $G_{n,k}$}
\label{s2.1}

We start with an elementary notation.
Let $n \ge 2$ and $1 \le k<n-1$.
We denote by $e_1,\dots,e_n$ the standard coordinate unit vectors in $\mathbb{R}^n$ and set
\[
\tau_0 :=\mathbb{R}^k = \mathbb{R}e_1 + \cdots + \mathbb{R}e_k= \mathbb{R}^k \times \{0_{n-k}\}
\subset {\mathbb R}^n,
\qquad
\mathbb{R}^{n-k} =\tau_0^\perp= \mathbb{R}e_{k+1} + \cdots + \mathbb{R}e_n\subset {\mathbb R}^n.
\]
Sometimes $\tau_0$ is called the reference $k$-dimensional subspace.  
\begin{definition}
The \emph{linear Grassmannian} $G_{n,k}$ is the manifold of all 
real
$k$-dimensional linear subspaces of $\mathbb{R}^n${\rm:}
\[
G_{n,k} := \{ V \subset \mathbb{R}^n \,:\, 
V \mbox{ is a subspace with }\dim V = k\}.
\]
\end{definition}
We will specify its structure of a manifold 
by means of the homogeneous space representation.

The orthogonal group ${\rm O}(n)$ acts transitively on the Grassmannian $G_{n,k}$ by
\[
g \cdot V := g(V), \qquad g \in {\rm O}(n), \ V \in G_{n,k}.
\]
Once we equip $G_{n,k}$ with the structure of  a manifold described shortly,
this action is smooth and preserves the dimension of subspaces.
We will view ${\rm O}(k) \times {\rm O}(n-k)$ as a subgroup of ${\rm O}(n)$ via
$(A,B) \mapsto A \oplus B=\begin{pmatrix}A&0\\0&B\end{pmatrix}$.
Let $g,h \in {\rm O}(n)$.
Remark that
$g \tau_0=h\tau_0$
if and only if
$g^{-1}h\tau_0=\tau_0$.
This means that 
$g^{-1}h \in {\rm O}(k) \times {\rm O}(n-k)$.
Therefore the stabilizer subgroup of $\tau_0$ under this action is
\[
\mathrm{Stab}(\tau_0)
:= \{ g \in {\rm O}(n) \,:\, g(\tau_0) = \tau_0 \}
= {\rm O}(k) \times {\rm O}(n-k).
\]

Consequently, the orbit--stabilizer theorem yields a canonical identification
\[
G_{n,k} \cong {\rm O}(n)/\mathrm{Stab}(\tau_0) \cong {\rm O}(n)/({\rm O}(k) \times {\rm O}(n-k)),
\]
endowing $G_{n,k}$ with the structure of a compact smooth homogeneous manifold.

More explicitly, the mapping
\[
\pi:{\rm O}(n) \to G_{n,k}, 
\qquad \pi(g) := g \cdot \tau_0
\]
is continuous and surjective, and satisfies
\[
\pi(g h) = \pi(g), \qquad h \in {\rm O}(k) \times {\rm O}(n-k).
\]
Therefore $\pi$ descends to a well-defined bijection
\[
\tilde{\pi}:
{\rm O}(n)/({\rm O}(k) \times {\rm O}(n-k)) \to G_{n,k},
\qquad
[g] \longmapsto g \cdot \tau_0=g(\tau_0).
\]
Equip $G_{n,k}$ with the structure of a manifold
so that this bijection  is a diffeomorphism.
A direct consequence that
the dimension of $G_{n,k}$ is
\[
\dim G_{n,k} = \frac{n(n-1)}{2}-\frac{k(k-1)}{2}-\frac{(n-k)(n-k-1)}{2} =k(n-k).
\]
Moreover, $G_{n,k}$ carries a unique ${\rm O}(n)$-invariant probability measure
$\gamma$,
induced from the Haar measure on ${\rm O}(n)$ via this quotient representation.
More precisely, the measure is given as follows{\rm:}
Let $\gamma$ be the probability Haar measure on the orthogonal group ${\rm O}(n)$. Then a natural ${\rm O}(n)$-invariant measure on $G_{n,k}$ is defined via 
$\pi$
as
\[
\mu_{G_{n,k}}(E) := \gamma(\pi^{-1}(E))
\]
for any measurable $E \subset G_{n,k}$.

\subsection{The Affine Grassmannian $\mathcal{G}_{n,k}$}
\label{s2.2a}

Let ${\mathcal G}_{n,k}$ denote the manifold of all nonoriented $k$-dimensional affine planes
$\tau$ in $\mathbb{R}^n$, where $1 \le k < n$.
The \emph{affine Grassmannian} $\mathcal{G}_{n,k}$ is the set of all $k$-dimensional affine planes (not necessarily through the origin) in $\mathbb{R}^n${\rm:}
\[
\mathcal{G}_{n,k} := \{\tau \subset \mathbb{R}^n \,:\, \tau = V + x, \ V \in G_{n,k},\ x \in \mathbb{R}^n \}.
\]

Each affine plane can be uniquely decomposed as
\[
\tau = V + x, \quad V \in G_{n,k},\quad x \in V^\perp,
\] where $V^\perp$ stands for the orthogonal complement
for $V \in G_{n,k}$.

The space ${\mathcal G}_{n,k}$ also has
a homogeneous space representation.
Recall first that
the \emph{Euclidean group} is
\[
M(n) := {\rm O}(n) \ltimes \mathbb{R}^n,
\]
with multiplication
\[
(T,y) \cdot (S,z) := (TS, T z + y), \quad T,S \in {\rm O}(n), \ y,z \in \mathbb{R}^n.
\]
It is noteworthy that $(T,y)$ generates an isometry on ${\mathbb R}^n${\rm:}
$I_{(T,y)}{\rm:}x \mapsto T x+y$.
Remark that $I_{(T,y)} \circ I_{(S,z)}=I_{(T,y) \cdot(S,z)}$.
The group acts transitively on $\mathcal{G}_{n,k}${\rm:}
\[
\mathcal{G}_{n,k}=(T,y) \cdot \tau_0 =I_{(T,y)}(\tau_0)= \{ T v + y \,:\, v \in \tau_0 \}.
\]
Let $d\gamma$ be the Haar measure on ${\rm O}(n)$ 
as before
and $dy$ the Lebesgue measure on $\mathbb{R}^n$. Then the Haar measure on $M(n)$ is
\[
d\mu_{M(n)}(T,y) = d\gamma(T) \, dy.
\]
Thus, for a measurable set $E \subset M(n)$,
the measure is given by
\[
\mu_{M(n)}(E)=\int_{M(n)}\chi_E(T,y)d\gamma(T)\,dy,
\]
where $\chi_E$ is the indicator function of $E$.

Let
\[
H := ({\rm O}(k) \times {\rm O}(n-k)) \ltimes \mathbb{R}^k.
\]
Any element of $H$ is written as $(A,B,a)$ with
$A \in {\rm O}(k)$, $B \in {\rm O}(n-k)$, and $a \in \mathbb{R}^k$.
We define an embedding
\[
\iota:H \hookrightarrow M(n)
\]
by
\[
\iota(A,B,a)
=
\left(
A \oplus B,
\begin{pmatrix}
a \\
0_{n-k}
\end{pmatrix}
\right)
=
\left(
\begin{pmatrix}
A & 0 \\
0 & B
\end{pmatrix},
\begin{pmatrix}
a \\
0_{n-k}
\end{pmatrix}
\right).
\]
Therefore,
the stabilizer of the reference plane $\tau_0 = \mathbb{R}^k \times \{0\}$ is
$H$
if we go through a similar argument as before.

Thus we can represent the affine Grassmannian as a homogeneous space{\rm:}
\[
\mathcal{G}_{n,k} \cong M(n)/H.
\]

The natural measure on the affine Grassmannian $\mathcal{G}_{n,k}$ is induced from $M(n)${\rm:}
\begin{equation}\label{eq:260112-1}
d\mu_{\mathcal{G}_{n,k}}(\tau) = d\mu_{\mathrm{G}_{n,k}}(V) \, dx, \quad \tau = V + x, \ x \in V^\perp.
\end{equation}
The precise meaning of (\ref{eq:260112-1}) is as follows{\rm:}
Let $A \subset \mathcal{G}_{n,k}$ be measurable. Then
\[
\mu_{\mathcal{G}_{n,k}}(A) = \int_{V \in G_{n,k}} \left( \int_{x \in V^\perp} \chi_A(V+x) \, dx \right) d\mu_{G_{n,k}}(V).
\]

Here is an example{\rm:}
\begin{example}[Lines in $\mathbb{R}^3$ ($k=1, n=3$)]
Let ${\mathbb R}{\mathbb P}^2$ denote $G_{3,1}$.
Then the
measure of a set of lines $
A$ is given by
\[
\mu_{\mathcal{G}_{3,1}}(A) = \int_{V \in \mathbb{RP}^2} \left( \int_{x \in V^\perp} \chi_A(V+x) \, dx \right) d\mu_{{\mathbb R}{\mathbb P}^2}(V).
\]
\end{example}

\subsection{The $k$-plane Radon and ridgelet transforms}
\label{s2.2}

The operators \eqref{eq:Ut} and \eqref{eq:Vt-star} may be viewed as
wavelet-like transforms associated with the $k$-plane Radon transform
\eqref{eq:k-plane-radon} and its dual.
We therefore refer to them as the \emph{$k$-plane ridgelet transforms}.

The asterisk ``$*$'' in $V_t^{*}$ and $W_t^*$ indicates that these operators act in the
reverse direction, mapping functions on ${\mathcal G}_{n,k}$ back to functions on
$\mathbb{R}^n$.

Usually, $w$ is generated by $u$ and $v$ via convolution in ${\mathbb R}^{n-k}$
as in (\ref{eq:def-w-convolution}).
We assume that $u$, $v$, and $w$ are radial integrable functions on $\mathbb{R}^{n-k}$.
By a slight abuse of notation, we write
\[
u(y)=u(|y|), \qquad
v(y)=v(|y|), \qquad
w(y)=w(|y|)
\]
for $y \in \mathbb{R}^{n-k}$.
We assume that $w$ satisfies the following condition, called
\emph{admissibility}
as in Definition \ref{def:admissible-pair}
\cite[p. 237]{Rubin2004}.

We move on to
the dual $k$-plane
transform $\check{\varphi}$ of a function
$\varphi:{\mathcal G}_{k,n} \to {\mathbb C}$ is defined by
\begin{equation}\label{eq:dual-k-plane-transform}
\check{\varphi}(x)
=
\int_{\mathrm{SO}(n)}
\varphi(x+\gamma\tau_0)\,d\gamma,
\qquad x\in\mathbb{R}^n.
\end{equation}

We recall an important example of
calculating $\hat f$ and $\check{\varphi}$.
\begin{lemma}{\rm \cite[Lemma 2.1]{Rubin2004}}\label{lem:abel-representation}
For $x\in\mathbb{R}^n$ and $\tau:=\zeta+x''\in {\mathcal G}_{n,k}$
with $x'' \in \zeta^\perp$, let
\begin{equation}\label{eq:r-s-def}
r=|x|=\operatorname{dist}(0,x),
\qquad
s=|x''|=\operatorname{dist}(0,\tau)=|\tau|
\end{equation}
denote the corresponding distances from the origin.
If $f$ and $\varphi$ are radial, that is,
$f(x):= f_0(r)$ and $\varphi(\tau):=\varphi_0(s)$, then
$\hat f(\tau)$ and $\check{\varphi}(x)$ admit the representations
\begin{equation}\label{eq:abel-hatf}
\hat f(\tau)
=
\sigma_{k-1}
\int_s^{\infty}
f_0(t)\,
(t^2-s^2)^{\frac{k}{2}-1}
\,t\,dt, \qquad \tau \in \mathcal{G}_{n,k},
\end{equation}
and
\begin{equation}\label{eq:abel-checkphi}
\check{\varphi}(x)
=
\frac{\sigma_{k-1}\sigma_{n-k-1}}{\sigma_{n-1}}
\,r^{2-n}
\int_0^{r}
\varphi_0(t)\,
(r^2-t^2)^{\frac{k}{2}-1}
\,t^{n-k-1}\,dt,
\end{equation}
provided that the integrals
$\displaystyle \int_s^{\infty}f_0(t)\,
(t^2-s^2)^{\frac{k}{2}-1}
\,t\,dt$ and $\displaystyle \int_0^{r} \varphi_0(t)\,
(r^2-t^2)^{\frac{k}{2}-1}
\,t^{n-k-1}\,dt$ are absolutely convergent.
\end{lemma}
We invoke \cite[Lemma 2.3]{Rubin2004}.

Equality \eqref{eq:abel-hatf} is a consequence of the bipolar coordinate change.
\[
x=T\omega+S\theta, \quad
T,
S \ge 0,\quad
\omega \in \tau \cap S^{n-1},\quad
\theta \in \tau^\perp \cap S^{n-1}.
\]
We denote
\begin{equation}\label{eq:inner-products}
(f_1,f_2)
=
\int_{\mathbb{R}^n}
f_1(x)f_2(x)\,dx,
\qquad
(\varphi_1,\varphi_2)^{\sim}
=
\int_{{\mathcal G}_{n,k}}
\varphi_1(\tau)\varphi_2(\tau)\,d\mu_{\mathcal{G}_{n,k}}(\tau).
\end{equation}
An important duality relation for
the transforms~\eqref{eq:k-plane-radon} and~\eqref{eq:dual-k-plane-transform} is
\begin{equation}\label{eq:duality}
(\hat f,\varphi)^{\sim} = (f,\check{\varphi}),
\end{equation}
provided that either one of the following holds:
\begin{enumerate}
\item
$f$ and $\varphi$ are non-negative functions.
\item
Either
 side is finite when $f$ and $\varphi$ are replaced by
$|f|$ and $|\varphi|$, respectively.
\end{enumerate}
See Rubin~\cite[(2.4)]{Rubin2004}
for the case where both sides in (\ref{eq:duality}) with $f$ and $\varphi$ replaced by
$|f|$ and $|\varphi|$
are finite.
An inspection of the proof of \cite[(2.4)]{Rubin2004} shows that
(\ref{eq:duality}) holds for  non-negative functions.
We also recall the following result from \cite[Lemma 2.3]{Rubin2004}.

\begin{lemma}\label{lem:convolution-identities}
Let $a(\cdot)$ and $b(\cdot)$ be measurable functions on $\mathbb{R}_{+}$.
Assume that
\begin{equation}\label{eq:260118-2}
\int_0^{r}
s^{n-k-1}
|a(s)|
(r^2-s^2)^{\frac{k}{2}-1}
\,ds < \infty,
\qquad r>0,
\end{equation}
and that
\begin{equation}\label{eq:260118-3}
\int_s^{\infty}
|b(r)|
(r^2-s^2)^{\frac{k}{2}-1}
\,r\,dr < \infty,
\qquad s>0.
\end{equation}
Define
\begin{equation}\label{eq:psi-def}
\psi(r)
=
c\, r^{2-n}
\int_0^{r}
s^{n-k-1}
a(s)\,
(r^2-s^2)^{\frac{k}{2}-1}
\,ds,
\qquad
c=\frac{\sigma_{k-1}\sigma_{n-k-1}}{\sigma_{n-1}}
\end{equation}
and
\begin{equation}\label{eq:h-def}
h(s)
=
\sigma_{k-1}
\int_s^{\infty}
b(r)\,
(r^2-s^2)^{\frac{k}{2}-1}
\,r\,dr.
\end{equation}
Then the following convolution identities hold:
\begin{enumerate}
\item[$(1)$]
Let $f \in L^1_{\mathrm{loc}}(\mathbb{R}^n)$ and let $x \in \mathbb{R}^n$ be such that
\begin{equation}\label{eq:260118-11}
\int_{0}^{\infty}
\int_{\{\,y\in\mathbb{R}^n : |x-y|\ge s\,\}}
|f(y)\, a(s)|\,
|x-y|^{2-n}\,
s^{n-k-1}\,
\left(|x-y|^{2}-s^{2}\right)^{k/2-1}
\, dy\, ds < \infty.
\end{equation}
Then
\begin{equation}\label{eq:identity-i}
\int_{{\mathcal G}_{n,k}}
\hat f(\tau)\,
a\!\left(|x-\tau|\right)
\,d\mu_{\mathcal{G}_{n,k}}(\tau)
=
\int_{\mathbb{R}^n}
f(y)\,
\psi\!\left(|x-y|\right)
\,dy.
\end{equation}
\item[$(2)$]
Let $\varphi \in L^1_{\mathrm{loc}}({\mathcal G}_{n,k})$ and let $x \in \mathbb{R}^n$ be such that
\begin{equation}\label{eq:260118-12}
\int_{{\mathcal G}_{n,k}}
\,\int_{|x-\tau|}^{\infty}
|\varphi(\tau)b(r)|
(r^2-|x-\tau|^2)^{\frac{k}{2}-1}
\,r\,dr
\,d\mu_{\mathcal{G}_{n,k}}(\tau) < \infty.
\end{equation}
Then
\begin{equation}\label{eq:identity-ii}
\int_{\mathbb{R}^n}
\check{\varphi}(y)\,
b\!\left(|x-y|\right)
\,dy
=
\int_{{\mathcal G}_{n,k}}
\varphi(\tau)\,
h\!\left(|x-\tau|\right)
\,d\mu_{\mathcal{G}_{n,k}}(\tau).
\end{equation}
\end{enumerate}
\end{lemma}

\begin{proof}
The proofs of (\ref{eq:identity-i}) and (\ref{eq:identity-ii}) rely on Fubini's theorem, 
Lemma~\ref{lem:abel-representation}, and (\ref{eq:inner-products}). 
The application of Fubini's theorem can be justified 
by (\ref{eq:260118-11}) and (\ref{eq:260118-12}), respectively. 
See also the proof of \cite[Lemma~12.29]{Rubin2015}.
\end{proof}

A clarifying remark concerning the proof of Lemma~\ref{lem:convolution-identities} is in order.

We consider
the convolution--backprojection method and ridgelet transforms.

Let $0<\alpha<n$.
Let $I_\alpha$ be the {\it fractional integral operator}
given by
\begin{equation}\label{eq:170426-99}
I_\alpha f(x)
:=
\int_{{\mathbb R}^n}\frac{f(y)}{|x-y|^{n-\alpha}}\,dy,
\quad x \in {\mathbb R}^n,
\end{equation}
for a non-negative measurable function $f:{\mathbb R}^n \to [0,\infty]$
or more generally for a measurable function $f:{\mathbb R}^n \to {\mathbb C}$.
We consider $I_\alpha f(x)$ as long as the integral makes sense.
We will need the following auxiliary estimate:

\begin{lemma}\label{lem:ball-integral-estimate}
Let $B \subset \mathbb{R}^n$ be a ball of radius $1$, and let $c_B$ denote its center.
Then, for all $0 < s \le 2$ and $y \in \mathbb{R}^n$,
\begin{equation}\label{eq:ball-integral-estimate}
\int_B
\frac{\chi_{(s,\infty)}(|x-y|)}
{|x-y|^{n-2}\bigl(|x-y|^2 - s^2\bigr)^{1/2}}
\, dx
\;\lesssim\;
\frac{1+s}{1+|y-c_B|^{\,n-1}}
\sim
\frac{1}{1+|y-c_B|^{\,n-1}},
\end{equation}
where the implicit constant is independent of $B$, $y$, and $s$.
\end{lemma}

\begin{proof}
We distinguish two cases according to the distance between $y$ and the center $c_B$.

\begin{enumerate}
\item[\emph{Case 1.}]
$|y-c_B| \ge 5$.

Recall that $s \le 2$.
For all $x \in B$, we have $|x-y| \gtrsim |y-c_B|\ge 5$
and
$|x-y| \ge|y-c_B|- |x-c_B|\ge4 \ge 2s$, and hence
\begin{equation}\label{eq:far-pointwise}
\frac{\chi_{(s,\infty)}(|x-y|)}
{|x-y|^{n-2}\bigl(|x-y|^2 - s^2\bigr)^{1/2}}
\;\lesssim\;
\frac{1}{1+|y-c_B|^{\,n-1}}.
\end{equation}
Integrating \eqref{eq:far-pointwise} over $B$ yields
\begin{equation}\label{eq:far-integral}
\int_B
\frac{\chi_{(s,\infty)}(|x-y|)}
{|x-y|^{n-2}\bigl(|x-y|^2 - s^2\bigr)^{1/2}}
\, dx
\;\lesssim\;
\frac{1}{1+|y-c_B|^{\,n-1}}.
\end{equation}

\item[\emph{Case 2.}]
$|y-c_B| \le 5$.

In this case, $B \subset B(y,6)$. Therefore,
\begin{align}
\int_B
\frac{\chi_{(s,\infty)}(|x-y|)}
{|x-y|^{n-2}\bigl(|x-y|^2 - s^2\bigr)^{1/2}}
\, dx
&\le
\int_{B(y,6)}
\frac{\chi_{(s,\infty)}(|x-y|)}
{|x-y|^{n-2}\bigl(|x-y|^2 - s^2\bigr)^{1/2}}
\, dx \nonumber\\
&\sim
\int_s^{6}
\frac{r\,dr}{\sqrt{r^2 - s^2}}
\;\lesssim\;
1, \label{eq:polar-estimate}
\end{align}
where we used polar coordinates centered at $y$ in the last step.
\end{enumerate}

Combining \eqref{eq:far-integral} and \eqref{eq:polar-estimate} completes the proof.
\end{proof}
\begin{lemma}\label{lem:ball-integral-estimate1}
Let $\beta \ge 1$.
Let $B \subset \mathbb{R}^n$ be a ball of radius $1$, and let $c_B$ denote its center.
Then, for all $s>2$, $\beta \ge 1$, and $y\in{\mathbb R}^n$, 
\begin{equation}\label{eq:ball-integral-estimate1}
\int_{B}
\frac{\chi_{(s,2s)}(|x-y|)}
{|x-y|^{n-2+\beta}\bigl(|x-y|^2-s^2\bigr)^{1/2}}
\,dx
\;\lesssim\;
\frac{1}{1+|y-c_B|^{n-1}},
\end{equation}
where the implicit constant is independent of $B$, $y$, and $s$.
\end{lemma}

\begin{proof}
We distinguish two cases according to the distance between $y$ and the center $c_B$.

\begin{enumerate}
\item[\emph{Case 1.}]
$|y-c_B| \le 10$.

In this case $B \subset B(y,12)$. Using polar coordinates centered at $y$, we obtain
\begin{align*}
\int_{B}
\frac{\chi_{(s,2s)}(|x-y|)}
{|x-y|^{n-2+\beta}\bigl(|x-y|^2-s^2\bigr)^{1/2}}
\,dx
&\,\lesssim
\int_s^{2s}
\frac{r^{n-1}}{r^{n-2+\beta}\sqrt{r^2-s^2}}
\,dr \nonumber \\
&\,\sim
\int_s^{2s}
\frac{r^{1-\beta}}{\sqrt{r^2-s^2}}
\,dr
\lesssim
s^{1-\beta}. \label{eq:near-estimate1}
\end{align*}
Since $s>2$ and $|y-c_B|\le 10$, we have
\[
s^{1-\beta}
\;\lesssim\;
\frac{1}{1+|y-c_B|^{n-1}}.
\]

\item[\emph{Case 2.}]
$|y-c_B| \ge 10$.

For $x \in B$ we have $|x-y| \sim |y-c_B|$, and hence
\begin{equation}\label{eq:far-pointwise1}
|x-y|^{-(n-2+\beta)}
\;\lesssim\;
|y-c_B|^{-(n-2+\beta)}.
\end{equation}
Therefore,
\begin{align}
\int_{B}
\frac{\chi_{(s,2s)}(|x-y|)}
{|x-y|^{n-2+\beta}\bigl(|x-y|^2-s^2\bigr)^{1/2}}
\,dx
&\lesssim
\frac{1}{|y-c_B|^{n-2+\beta}}
\int_{B}
\frac{\chi_{(s,2s)}(|x-y|)}
{\sqrt{|x-y|^2-s^2}}
\,dx. \label{eq:far-reduction1}
\end{align}

Let
\[
E := \left\{ \frac{x-y}{|x-y|} : x \in B \right\} \subset S^{n-1}.
\]
A standard geometric argument yields
\[
{\mathcal H}^{n-1}(E) \lesssim |y-c_B|^{-(n-1)}.
\]
Denote by ${\mathcal H}^{n-1}$ the $n-1$-dimensional Hausdorff measure.
Using polar coordinates $(r,\omega)$ with $\omega \in E$, we obtain
\begin{align}
\int_{B}
\frac{\chi_{(s,2s)}(|x-y|)}
{\sqrt{|x-y|^2-s^2}}
\,dx
\nonumber
&\lesssim
\int_E
\int_{s}^{2s}
\frac{r^{n-1}\chi_{(|y-c_B|-1,|y-c_B|+1)}(r)}{\sqrt{r^2-s^2}}
\,dr\,d{\mathcal H}^{n-1}(\omega)  \\
&\lesssim
{\mathcal H}^{n-1}(E)\,|y-c_B|^{n-1}
\lesssim
1. \label{eq:angular-estimate1}
\end{align}
Combining \eqref{eq:far-reduction1}
and
\eqref{eq:angular-estimate1} gives
\[
\int_{B}
\frac{\chi_{(s,2s)}(|x-y|)}
{|x-y|^{n-2+\beta}\bigl(|x-y|^2-s^2\bigr)^{1/2}}
\,dx
\;\lesssim\;
\frac{1}{|y-c_B|^{n+\beta-2}}
\;\lesssim\;
\frac{1}{1+|y-c_B|^{n-1}},
\]
since $s>2$, $\beta \ge 1$ and $|y-c_B|\ge 10$.
\end{enumerate}

This completes the proof.
\end{proof}

\begin{remark}
\label{rem 20260306-1}
Let $k=1,2,\ldots,n-1$.
Let $a,b:(0,\infty)\to\mathbb{R}$ be integrable functions such that
\[
I_0(a) := \int_0^\infty s^{\,n-k-1} |a(s)|\, ds < \infty, 
\qquad
I_0(b) := \int_0^\infty s^{\,n-k-1} |b(s)|\, ds < \infty.
\]
Assume additionally that
\[
I_\beta(a) := \int_0^\infty s^{\beta+n-k-1} |a(s)|\, ds < \infty, 
\qquad
I_\beta(b) := \int_0^\infty s^{\beta+n-k-1} |b(s)|\, ds < \infty
\]for some $\beta>k$. 

Identify $a$ and $b$ with the corresponding radial functions on $\mathbb{R}^{\,n-k}$, and define
\[
j(x) := a * b(x), \qquad 
j_+(x) := |a| * |b|(x), \qquad x \in \mathbb{R}^{\,n-k}.
\]
Note that $j$ and $j_+$ are radial.
We also identify $j$ and $j_+$
naturally with the associated functions on $(0,\infty)$:
$j(t)=j(t,0,0,\ldots,0)$
and
$j_+(t)=j_+(t,0,0,\ldots,0)$
for $t>0$.
\begin{enumerate}
\item 
We have
\begin{align*}
I_0(j_+) := \int_0^\infty s^{\,n-k-1} j_+(s)\, ds
&\sim \int_{\mathbb{R}^{\,n-k}} j_+(x)\, dx \\
&= \|a\|_{L^1({\mathbb R}^{n-k})} \|b\|_{L^1({\mathbb R}^{n-k})} \\
&\sim I_0(a) I_0(b) < \infty.
\end{align*}
Meanwhile
$|j(s)| \le j_+(s)$ for all $s>0$.
by the triangle inequality.
Consequently, 
\[
I_0(j) := \int_0^\infty s^{\,n-k-1} |j(s)|\, ds < \infty.
\]

\item
Condition~\eqref{eq:260118-11} holds almost everywhere for all
$f \in L^1(\mathbb{R}^n) + L^p(\mathbb{R}^n)$ with $p < \frac{n}{k}$.

We deal with the case $k \ge 2$.
Since
\[
(|x-y|^2 - s^2)^{\frac{k}{2}-1} \le |x-y|^{k-2},
\]
we obtain
\begin{align*}
&\int_0^\infty
 \int_{\{|x-y|\ge s\}}
 |f(y)\,a(s)|\, |x-y|^{2-n}
 s^{\,n-k-1} (|x-y|^2 - s^2)^{\frac{k}{2}-1}
 \,dy\,ds \\
&\qquad\le
 \int_0^\infty
 \int_{\mathbb{R}^n}
 \frac{|f(y)|}{|x-y|^{\,n-k}}
 s^{\,n-k-1} |a(s)|
 \,dy\,ds \\
&\qquad=
 I_0(a)\, I_k[|f|](x).
\end{align*}
By the Hardy--Littlewood--Sobolev theorem,
\[
I_k[|f|]
\in
L^{\frac{n}{n-k}}(\mathbb{R}^n)
+
L^{\frac{np}{n-kp}}(\mathbb{R}^n),
\]
hence $I_k[|f|](x) < \infty$ for almost every $x \in \mathbb{R}^n$.
Therefore, condition~\eqref{eq:260118-11} is satisfied almost everywhere.

The case $k=1$ is more difficult to handle.
Proceeding as above, by replacing $s$ with $2s$, we have
\begin{align*}
&\int_0^\infty
 \int_{\{|x-y|\ge 2s\}}
 |f(y)\,a(s)|\, |x-y|^{2-n}
 s^{\,n-2} (|x-y|^2 - s^2)^{-1/2}
 \,dy\,ds \\ 
&\qquad\le2
 \int_0^\infty
 \int_{\mathbb{R}^n}
 \frac{|f(y)|}{|x-y|^{\,n-1}}
 s^{\,n-2} |a(s)|
 \,dy\,ds \\ 
&\qquad=2
 I_0(a)\, I_1[|f|](x),
\end{align*}
using the estimate
\[
(|x-y|^2 - s^2)^{-1/2} \le 2|x-y|^{-1},
\qquad
s < \tfrac12 |x-y|.
\]

It remains to show that
\begin{align}\label{eq:260226-95}
&\int_0^\infty
 \int_{\{2s \ge |x-y| \ge s\}}
 |f(y)\,a(s)|\, |x-y|^{2-n}
 s^{\,n-2} (|x-y|^2 - s^2)^{-1/2}
 \,dy\,ds
 < \infty
\end{align}
for almost all $x \in {\mathbb R}^n$.
This reduces to proving finiteness of the two integrals
\begin{align*}
&\int_0^\infty
 \int_{\{2s \ge |x-y| \ge s \ge 2\}}
 |f(y)\,a(s)|\, |x-y|^{2-n}
 s^{\,n-2} (|x-y|^2 - s^2)^{-1/2}
 \,dy\,ds, 
\\[0.5ex]
&\int_0^\infty
 \int_{\{4 \ge 2s \ge |x-y| \ge s\}}
 |f(y)\,a(s)|\, |x-y|^{2-n}
 s^{\,n-2} (|x-y|^2 - s^2)^{-1/2}
 \,dy\,ds.
\end{align*}

Let $B$ be a ball of radius $1$.
We seek to show that the integration of the left-hand side
of \eqref{eq:260226-95} is finite.
Since
$f \in L^1({\mathbb R}^n){+L^p({\mathbb R}^n)}$ and $I_1[|f|](x)<\infty$ almost everywhere,
we have
\[
\int_{{\mathbb R}^n}\frac{|f(y)|}{1+|y-c_B|^{n-1}}\,dy<\infty
\]
according to \cite[Lemma 180]{book}.
By Lemma~\ref{lem:ball-integral-estimate},
\begin{align*}
&\int_B
 \int_0^\infty
 \int_{\{ 4\ge 2s \ge |x-y| \ge s\}}
 |f(y)\,a(s)|\, |x-y|^{2-n}
 s^{\,n-2} (|x-y|^2 - s^2)^{-1/2}
 \,dy\,ds\,dx \\ 
&\qquad\lesssim
I_0(a)
 \int_{\mathbb{R}^n}
 \frac{|f(y)|}{1+|y-c_B|^{\,n-1}}
 \,dy
 < \infty.
\end{align*}
Recall that $\beta \ge 1$.
Likewise, by Lemma~\ref{lem:ball-integral-estimate1},
\begin{align*}
&\int_B
 \int_0^\infty
 \int_{\{2s \ge |x-y| \ge s \ge 2\}}
 |f(y)\,a(s)|\, |x-y|^{2-n}
 s^{\,n-2} (|x-y|^2 - s^2)^{-1/2}
 \,dy\,ds\,dx \\ 
&\qquad\lesssim
 \int_B
 \int_0^\infty
 \int_{\{2s \ge |x-y| \ge s \ge 2\}}
 |f(y)\,a(s)|\, |x-y|^{-n+2-\beta}
 s^{n-2+\beta} (|x-y|^2 - s^2)^{-1/2}
 \,dy\,ds\,dx \\ 
&\qquad\lesssim
I_\beta(a)
 \int_{\mathbb{R}^n}
 \frac{|f(y)|}{1+|y-c_B|^{\,n-1}}
 \,dy
 < \infty.
\end{align*}

Consequently, for almost every $x \in \mathbb{R}^n$,
\begin{align*}
&\int_0^\infty
 \int_{\{2s \ge |x-y| \ge s\}}
 |f(y)\,a(s)|\, |x-y|^{2-n}
 s^{\,n-2} (|x-y|^2 - s^2)^{-1/2}
 \,dy\,ds
 < \infty.
\end{align*}
This proves that
\[
\int_0^\infty
\int_{\{|x-y|\ge s\}}
|f(y)\,a(s)|\, |x-y|^{2-n}
s^{\,n-k-1} (|x-y|^2 - s^2)^{\frac{k}{2}-1}
\,dy\,ds
\]
is finite for almost every $x$ in the case $k=1$.
\item
Arguing similarly,
\begin{align*}
I_\beta(j_+) &:= \int_0^\infty s^{\,\beta+n-k-1} j_+(s)\, ds 
\sim \int_{\mathbb{R}^{\,n-k}} |x|^\beta j_+(x)\, dx \\
&\lesssim \int_{\mathbb{R}^{\,n-k}} \int_{\mathbb{R}^{\,n-k}} (|x-y|^\beta + |y|^\beta) |a(y)b(x-y)|\, dy\, dx \\
&\sim I_\beta(a) I_0(b) + I_0(a) I_\beta(b) < \infty.
\end{align*}
Hence, if $a$ and $b$ are admissible, so is $a*b$.

\item
Let $a,b$ be admissible functions.

\begin{enumerate}
\item
For $t>0$ and measurable $f:\mathbb{R}^n \to \mathbb{C}$, define the generalized projection operators (smoothed $k$-plane Radon transforms)
\begin{align}
A_t f(\tau) &= t^{\,k-n} \int_{\mathbb{R}^n} f(x)\, a\!\left(\frac{|x-\tau|}{t}\right) dx, \\
A_t^+ f(\tau) &= t^{\,k-n} \int_{\mathbb{R}^n} f(x)\, \left|a\!\left(\frac{|x-\tau|}{t}\right)\right| dx, \\
J_t f(\tau) &= t^{\,k-n} \int_{\mathbb{R}^n} f(x)\, j\!\left(\frac{|x-\tau|}{t}\right) dx, \\
J_t^+ f(\tau) &= t^{\,k-n} \int_{\mathbb{R}^n} f(x)\, j_+\!\left(\frac{|x-\tau|}{t}\right) dx,
\end{align}
for $\tau \in {\mathcal G}_{n,k}$.

\item
For $t>0$ and measurable $\varphi:{\mathcal G}_{n,k}\to\mathbb{C}$, define the generalized backprojection (dual $k$-plane) operators
\begin{align}
B_t^* \varphi(x) &= t^{\,k-n} \int_{{\mathcal G}_{n,k}} \varphi(\tau)\, b\!\left(\frac{|x-\tau|}{t}\right) d\mu_{{\mathcal G}_{n,k}}(\tau), \\
B_t^{*,+} \varphi(x) &= t^{\,k-n} \int_{{\mathcal G}_{n,k}} \varphi(\tau)\, \left|b\!\left(\frac{|x-\tau|}{t}\right)\right| d\mu_{{\mathcal G}_{n,k}}(\tau),
\end{align}
for $x \in \mathbb{R}^n$.

Arguing as in Rubin~\cite[(3.20)]{Rubin2004}, we have
\[
B_t^{*,+} A_t^+ [|f|](x) = J_t^+ \widehat{|f|}(x),
\]
which is finite for almost all $x \in \mathbb{R}^n$ if $f \in L^1({\mathbb R}^n) + L^p({\mathbb R}^n)$, $p < \frac{n}{k}$. In particular, $A_t f$ satisfies \eqref{eq:260118-12} almost everywhere, and
\[
B_t^* A_t f = J_t^+ \hat{f}
\]
for all $f \in L^1({\mathbb R}^n) + L^p({\mathbb R}^n)$ with $p < \frac{n}{k}$.
\end{enumerate}
\end{enumerate}
\end{remark}
\subsection{An observation on an averaging lemma}
\label{s2.5}

In many function spaces arising in harmonic analysis, such as
variable exponent Lebesgue spaces
or weighted spaces, the norm is not translation invariant,
and convolution with a general $L^1$ kernel cannot be treated by classical
arguments based on Young's inequality.
Nevertheless, these spaces are typically stable under local averaging, namely,
under convolution with normalized characteristic functions of balls.
The purpose of the following lemma is to show that this weaker form of
averaging stability already suffices to control convolution with radial
decreasing kernels and to construct an approximation of the identity.
The proof relies on the layer--cake representation
of the function $K$, which reduces the problem
to finite superpositions of ball averages.
Recall that $c_n$ is a constant in \eqref{eq:260206-21}.
\begin{lemma}\label{lem:260119-11}
Let $X \in \mathcal A$ be a Banach lattice of measurable functions on
$\mathbb R^n$ satisfying the assumptions of
Theorem~\ref{thm:wavelet-reconstruction}.
In particular, assume that $C_{\mathrm c}(\mathbb R^n)$ is dense in $X$.

Let $K:\mathbb R^n \to [0,\infty)$ be a radial, decreasing, integrable function
and define
\[
K_\varepsilon(x) := \varepsilon^{-n} K(x/\varepsilon).
\]
Then the following statements hold:
\begin{enumerate}
\item[$(1)$]
For all $f \in X$,
\[
\|K * f\|_X
\le
c_n [X]_{\mathcal A}\, \|K\|_{L^1(\mathbb R^n)}\, \|f\|_X .
\]
\item[$(2)$]
For all $f \in X$,
\[
\lim_{\varepsilon \downarrow 0}
\|K_\varepsilon * f - \|K\|_{L^1} f\|_X = 0 .
\]
\end{enumerate}
\end{lemma}

Lemma~\ref{lem:260119-11}(1) may be viewed as an $n$-dimensional analogue of a
result of Vinogradov \cite[Lemma~2.2]{Vinogradov2024}.
Although the proof amounts to a careful inspection of that argument, we include
it here for the sake of completeness.

\begin{proof}
We may assume that $K \neq 0$; otherwise the conclusion is trivial.
Because $K$ is radial and decreasing, there exists a decreasing function
$\kappa : [0,\infty) \to [0,\infty)$ such that $K(x) = \kappa(|x|)$. 

For any $x\in{\mathbb R}^n$, we have 
\[
K(x) = \int_0^{\kappa(|x|)} 1\,dt = \int_0^{\infty} \chi_{ [0,\kappa(|x|)]} (t)\,dt. 
\]
For $N\in{\mathbb N}$, define 
\[
K_N(x) = 2^{-N} \sum_{l=1}^{N \cdot 2^N}  \chi_{ [0,\kappa(|x|)]}(2^{-N}l) = 2^{-N} \sum_{l=1}^{N \cdot 2^N} \chi_{\{y\,:\, l \le  2^N\kappa(|y|)\}} (x), 
\] 
which is the right-endpoint Riemann sum for $\displaystyle \int_0^{N} \chi_{ [0,\kappa(|x|)]} (t)\,dt$. 
Since 
\[
 \sum_{l=1}^{N \cdot 2^N} \chi_{\{y\,:\, l \le 2^N\kappa(|y|)\}} (x) = \#\{l\in\{1,2,\ldots,N \cdot 2^N\}\,:\, l  \le
 2^N\kappa(|x|)\}, 
\]
it follows that 
\[
K_N(x) = 2^{-N} \min\{ [2^N \kappa(|x|)], 2^{N}\cdot N\} = 2^{-N} [2^N\min\{ \kappa(|x|), N\}]. 
\]

By the elementary properties of the floor function $[\,\cdot\,]$, we obtain $K_N(x) \le K_{N+1}(x)$ and 
\[
\min\{ \kappa(|x|), N\} -2^{-N} <K_N(x) \le \min\{ \kappa(|x|), N\}. 
\]
Therefore, $K_N(x) \uparrow K(x)$ for any $x\in{\mathbb R}^n$. 
Moreover, since $0\le K_N(x)\le K(x)$ and $K\in L^1(\mathbb R^n)$,
the dominated convergence theorem implies
\[
\|K_N-K\|_{L^1(\mathbb R^n)}\to 0.
\]

To prove (1) and (2), we define 
\[
r(t) := \sup\{ r>0 : \kappa(r) \ge t \}.
\]
Then 
\[
K_N(x) =  2^{-N} \sum_{l=1}^{N \cdot 2^N} \chi_{B(0, r(2^{-N}l))} (x)
\]
almost everywhere. 
In particular, for each $N$, $K_N$ is of the form 
\[
\sum_{l=1}^M a_l \chi_{B(0, r_l)}, \ \ a_l \ge 0. 
\]
Moreover, if $N<N'$,  the non-negative function $K_{N'}-K_N$ is again of this form.

Once assertions (1) and (2) are proved with $K$ replaced by $K_N$, the passage
to the limit $N \to \infty$ follows easily.
Indeed, for $N' > N$, the difference $K_{N'} - K_N$ has the same structure as
$K_N$, and
\[
\lim_{N' \to \infty} \|K_{N'} - K_N\|_{L^1} = 0 .
\]
By the Fatou property of $X$,
\[
\|(K - K_N) * f\|_X
\le
\liminf_{N' \to \infty} \|(K_{N'} - K_N) * f\|_X
\le
c_n [X]_{\mathcal A}\, 
\lim_{N' \to \infty}
\|K_{N'} - K_N\|_{L^1}\, \|f\|_X ,
\]
for all $f \in X$.
Thus it suffices to prove both assertions for kernels of the form
\[
K = \sum_{l=1}^N a_l\, \chi_{B(0,r_l)} .
\]

\begin{enumerate}
\item
Let $f \in X$ with $f \ge 0$.
Then
\[
K * f
=
\sum_{l=1}^N a_l\, (\chi_{B(0,r_l)} * f).
\]
Since $X \in \mathcal A$, we have by \eqref{eq:260206-21}
\[
\Bigl\|
\frac{1}{|B(0,r)|}\, \chi_{B(0,r)} * f
\Bigr\|_X
\le
c_n [X]_{\mathcal A}\, \|f\|_X .
\]
Consequently,
\[
\|\chi_{B(0,r_l)} * f\|_X
\le
c_n [X]_{\mathcal A}\, |B(0,r_l)|\, \|f\|_X .
\]
Using positivity and the triangle inequality in $X$, we obtain
\begin{align*}
\|K * f\|_X
&\le
\sum_{l=1}^N a_l\, \|\chi_{B(0,r_l)} * f\|_X\\
&\le
c_n [X]_{\mathcal A}
\sum_{l=1}^N a_l\, |B(0,r_l)|\, \|f\|_X\\
&=
c_n [X]_{\mathcal A}\, \|K\|_{L^1}\, \|f\|_X .
\end{align*}
\item
Let
\[
K = \sum_{l=1}^N a_l\, \chi_{B(0,r_l)}
\]
as in Step~1, and define $K_\varepsilon(x) = \varepsilon^{-n} K(x/\varepsilon)$.
Then
\[
K_\varepsilon(x)
=
\varepsilon^{-n}\sum_{l=1}^N a_l\, \chi_{B(0,\varepsilon r_l)}(x).
\]
By part~(1), the operators $f \mapsto K_\varepsilon * f$ are uniformly bounded
on $X$.
Since $C_{\mathrm c}(\mathbb R^n)$ is dense in $X$, it suffices to consider
$f \in C_{\mathrm c}(\mathbb R^n)$.
The general case $f \in X$ follows by density and the uniform bound established
in part~(1).

For such $f$,
\[
\frac{1}{|B(0,\varepsilon r)|}\,
\chi_{B(0,\varepsilon r)} * f
\to f
\quad \text{uniformly as } \varepsilon \downarrow 0 .
\]
Since
\[
\sum_{l=1}^N a_l\, |B(0,r_l)|
=
\|K\|_{L^1}
=
\|K_{\varepsilon}\|_{L^1}, 
\]
it follows that
\[
\lim_{\varepsilon \downarrow 0}
\sup_{x \in \mathbb R^n}
\left|
K_\varepsilon * f(x)
-
\|K_\varepsilon\|_{L^1}\, f(x)
\right|
=
\lim_{\varepsilon \downarrow 0}
\sup_{x \in \mathbb R^n}
\left|
K_\varepsilon * f(x)
-
\sum_{l=1}^N a_l\, |B(0,r_l)|\, f(x)
\right|
= 0 .
\]
Thus,
we conclude that $K_\varepsilon * f \to \|K\|_{L^1} f$ uniformly.
For $\varepsilon$ sufficiently small, the supports of $K_\varepsilon * f$
are contained in a fixed compact set $L$.
Using the continuous embedding $C_{\mathrm c}(L) \hookrightarrow X$,
which is a consequence of $({\rm L})$ and $({\rm BSi})$, we obtain
\[
\|K_\varepsilon * f - \|K\|_{L^1} f\|_X \to 0 .
\]
\end{enumerate}

\end{proof}

\begin{lemma}\label{lem:260119-81}
Let $w$ be an admissible function and define
\[
\psi(r)
=
c_{k,n}\,r^{2-n}
\int_0^{r}
s^{n-k-1}
w(s)\,
(r^2-s^2)^{\frac{k}{2}-1}
\,ds 
\]
for $r>0$. 
Here $c_{k,n}$ is the same constant as in \eqref{eq:psi-w-relation}. 
\begin{enumerate}
\item[$(1)$]
Let $k \ge 2$. Then
\[
|\psi(r)| = O(r^{k-n})
\qquad\text{as } r \downarrow 0 .
\]

\item[$(2)$]
Let $k \ge 1$. Then for every $\varepsilon>0$,
\begin{equation}\label{eq:260119-712}
\int_{\varepsilon}^{\infty}
\left|
\psi\!\left(\frac{|x|}{t}\right)
\right|
\frac{dt}{t^{n+1}}
\lesssim
\frac{C_\varepsilon}{|x|^{\,n-k}} .
\end{equation}
\end{enumerate}
\end{lemma}
\begin{proof}
\begin{enumerate}
\item
Let $k \ge 2$.
Since $0 \le s \le r$, we have
\[
(r^2 - s^2)^{\frac{k}{2}-1} \le r^{k-2}.
\]
Hence, by the triangle inequality,
\begin{align*}
|\psi(r)|
&\le
c_{k,n}\, r^{2-n}
\int_0^{r}
s^{n-k-1}
|w(s)|
(r^2 - s^2)^{\frac{k}{2}-1}
\,ds \\
&\le
c_{k,n}\, r^{k-n}
\int_0^{r}
s^{n-k-1}
|w(s)|
\,ds \\
&\le
c_{k,n}\, r^{k-n}
\int_0^{\infty}
s^{n-k-1}
|w(s)|
\,ds .
\end{align*}
Since $w$ is admissible, the last integral is finite, which proves
\[
|\psi(r)| = O(r^{k-n})
\quad \text{as } r \downarrow 0 .
\]

\item
We prove \eqref{eq:260119-712}. By definition of $\psi$ and
the triangle inequality,
\begin{align*}
\int_{\varepsilon}^{\infty}
\left|
\psi\!\left(\frac{|x|}{t}\right)
\right|
\frac{dt}{t^{n+1}}
&\lesssim
\int_{\varepsilon}^{\infty}
\left(\frac{|x|}{t}\right)^{2-n}
\int_0^{\frac{|x|}{t}}
s^{n-k-1}
|w(s)|
\left(
\left(\frac{|x|}{t}\right)^2 - s^2
\right)^{\frac{k}{2}-1}
\,ds\,
\frac{dt}{t^{n+1}} .
\end{align*}
Interchanging the order of integration yields
\begin{align*}
\int_{\varepsilon}^{\infty}
\left|
\psi\!\left(\frac{|x|}{t}\right)
\right|
\frac{dt}{t^{n+1}}
&\lesssim
|x|^{2-n}
\int_0^{\frac{|x|}{\varepsilon}}
s^{n-k-1}
|w(s)|
\int_{\varepsilon}^{\frac{|x|}{s}}
\frac{1}{t^{3}}
\left(
\left(\frac{|x|}{t}\right)^2 - s^2
\right)^{\frac{k}{2}-1}
\,dt\, ds .
\end{align*}

If $k \ge 2$, then
\[
\left(
\left(\frac{|x|}{t}\right)^2 - s^2
\right)^{\frac{k}{2}-1}
\le
\left(\frac{|x|}{t}\right)^{k-2},
\]
and the inner integral is bounded by
\[
\int_{\varepsilon}^{\infty}
\frac{dt}{t^{k+1}}
\lesssim
\varepsilon^{-k}.
\]
Hence,
arguing as before, we have
\[
\int_{\varepsilon}^{\infty}
\left|
\psi\!\left(\frac{|x|}{t}\right)
\right|
\frac{dt}{t^{n+1}}
\lesssim
\frac{1}{|x|^{\,n-k}} .
\]

If $k = 1$, then for $0 < s < \frac{|x|}{\varepsilon}$ we estimate
\begin{align*}
\int_{\varepsilon}^{\frac{|x|}{s}}
\frac{1}{t^3}
\left(
\left(\frac{|x|}{t}\right)^2 - s^2
\right)^{-\frac12}
\,dt
&=
\left[
-\frac{1}{|x|^2}
\left(
\left(\frac{|x|}{t}\right)^2 - s^2
\right)^{\frac12}
\right]_{\varepsilon}^{\frac{|x|}{s}} \\
&=
\frac{1}{|x|^2}
\left(
\left(\frac{|x|}{\varepsilon}\right)^2 - s^2
\right)^{\frac12}
\le 
\frac{1}{\varepsilon} \cdot
\frac{1}{|x|}.
\end{align*}
Therefore,
\[
\int_{\varepsilon}^{\infty}
\left|
\psi\!\left(\frac{|x|}{t}\right)
\right|
\frac{dt}{t^{n+1}}
\lesssim 
\frac{1}{\varepsilon} \cdot
|x|^{1-n}
\int_0^{\infty}
s^{n-2}
|w(s)|
\,ds
\lesssim 
\frac{1}{\varepsilon} \cdot
\frac{1}{|x|^{\,n-1}} .
\]
This completes the proof.
\end{enumerate}
\end{proof}

\section{Proof}
\label{s3}

\subsection{Proof of Theorem \ref{thm:wavelet-reconstruction}}
\label{s3.1}
The almost everywhere convergence follows from the assumption
$f \in L^{1}+L^{p}$.
Indeed, the pointwise convergence has already been established in
\cite{Rubin2004}.
Therefore, we restrict our attention to the convergence in norm.

Fix $t>0$.
Observe that
\begin{align*}
\lefteqn{
c_{k,n}\,r^{2-n}
\int_0^{r}
s^{n-k-1}
w\!\left(\frac{s}{t}\right)\,
(r^{2}-s^{2})^{\frac{k}{2}-1}
\,ds
}\\
&=
c_{k,n}\,t\,r^{2-n}
\int_0^{\frac{r}{t}}
(s t)^{n-k-1}
w\!\left(s\right)\,
(r^{2}-s^{2}t^2)^{\frac{k}{2}-1}
\,ds\\
&=
c_{k,n}\,\left(\frac{r}{t}\right)^{2-n}
\int_0^{\frac{r}{t}}
s^{n-k-1}
w(s)\,
\left(\frac{r^{2}}{t^{2}}-s^{2}\right)^{\frac{k}{2}-1}
\,ds\\
&=\psi\!\left(\frac{r}{t}\right).
\end{align*}

Since $f \in L^{1}+L^{p}$, \eqref{eq:260118-11} is satisfied with
$a = w(\cdot/t)$. 
See Remark \ref{rem 20260306-1}.
Hence, by \eqref{eq:psi-w-relation} and \eqref{eq:identity-i}, we obtain
\begin{equation}\label{eq:W-star-hatf}
W_t^{*}\hat f(x)
=
\frac{1}{t^{n}}
\int_{\mathbb{R}^{n}}
f(y)\,
\psi\!\left(\frac{|x-y|}{t}\right)
\,dy.
\end{equation}

Since we are assuming that $\psi$ admits a decreasing majorant,
we have the desired result by Lemma \ref{lem:260119-11}.
\subsection{Proof of Theorem \ref{thm:wavelet-reconstruction1}}
\label{s3.1}
The proof of Theorem \ref{thm:wavelet-reconstruction1}
follows the line of \cite[Theorem 3.2]{Rubin2004}.
For the sake of self-containedness, we supply the proof.
For $\varepsilon>0$, define
\begin{equation}\label{eq:tilde-psi-epsilon}
\tilde\psi_{\varepsilon}(x)
:=
\int_{\varepsilon}^{\infty}
\psi\!\left(\frac{|x|}{t}\right)
\,\frac{dt}{t^{n+1}}
=
\frac{1}{\varepsilon^{n}}
\tilde\psi\!\left(\frac{x}{\varepsilon}\right),
\end{equation}
where $\tilde\psi$ is given by \eqref{eq:260119-26}.
We first
verify the convolution
representation 
\begin{equation}\label{eq:epsilon-convolution}
\int_{\varepsilon}^{\infty} W_t^* \hat f(x) \,\frac{dt}{t}
=
\int_{\mathbb{R}^n} f(y)\, \tilde\psi_\varepsilon(x-y)\,dy,
\quad x\in \mathbb{R}^n, \; \varepsilon>0.
\end{equation}
For almost every $x\in{\mathbb R}^n$,  $I_k(|f|)(x)<\infty$
by the Hardy--Littlewood--Sobolev theorem.
By Lemma~\ref{lem:260119-81},
\begin{equation}\label{eq:260119-71}
\int_{\varepsilon}^{\infty}
\left|
\psi\!\left(\frac{|x|}{t}\right)
\right|
\frac{dt}{t^{n+1}}
\lesssim
\frac{C_\varepsilon}{|x|^{\,n-k}}.
\end{equation}
Hence, for almost all $x \in {\mathbb R}^n$,
Fubini's theorem applies, and using \eqref{eq:W-star-hatf} we compute
\begin{align*}
\int_{\varepsilon}^{\infty}
W_t^*\hat f(x)\,\frac{dt}{t}
&=
\int_{\mathbb R^n}
f(y)
\left(
\int_{\varepsilon}^{\infty}
\psi\!\left(\frac{|x-y|}{t}\right)
\frac{dt}{t^{n+1}}
\right) dy 
=
\int_{\mathbb R^n}
f(y)\,
\tilde\psi_\varepsilon(x-y)\,dy.
\end{align*}
This proves \eqref{eq:epsilon-convolution}.

Let
\[
c_1
:=
\frac{\pi^{k/2}\sigma_{n-k-1}}{2\sigma_{n-1}}.
\]
We next claim 
\begin{equation}\label{eq:260129-11}
\tilde\psi(x)
=c_1
|x|^{1-n}
\lambda(|x|),
\end{equation}
By the change of variables $u=|x|/t$ in \eqref{eq:260119-26}, we obtain
\[
\tilde\psi(x)
=
|x|^{-n}
\int_0^{|x|}
\psi(u)\,u^{\,n-1}\,du.
\]
Substituting the above expression for $\psi$
into $\tilde\psi$ and applying Fubini's theorem,
we arrive at
\[
\tilde\psi(x)
=
c_{k,n}\,|x|^{-n}
\int_0^{|x|}
u
\left(
\int_0^u
s^{\,n-k-1} w(s)\,(u^2-s^2)^{\frac{k}{2}-1}
\,ds
\right) du.
\]

Introduce the quadratic substitution $r=s^2$, $ds=\frac{dr}{2\sqrt r}$,
and define $\widetilde w$ by \eqref{eq:260128-1}, so that
$s^{\,n-k-1} w(s)\,ds = \frac12 \widetilde w(r)\,dr$.
Then
\[
\tilde\psi(x)
=
\frac{c_{k,n}}{2}
|x|^{-n}
\int_0^{|x|}
u
\left(
\int_0^{u^2}
(u^2-r)^{\frac{k}{2}-1}
\widetilde w(r)\,dr
\right) du.
\]Changing variables $v=u^2$ yields
\begin{align*}
\tilde\psi(x)
&=
\frac{c_{k,n}}{4}
|x|^{-n}
\int_0^{|x|^2}
\left(
\int_0^v
(v-r)^{\frac{k}{2}-1}
\widetilde w(r)\,dr
\right) dv \\
&=
\frac{c_{k,n}}{4}
|x|^{-n}
\int_0^{|x|^2}
\left(
\int_r^{|x|^2}
(v-r)^{\frac{k}{2}-1}
\widetilde w(r)\,dv
\right) dr \\
&=
\frac{c_{k,n}}{2k}
|x|^{-n}
\int_0^{|x|^2}
(|x|^2-r)^{\frac{k}{2}}
\widetilde w(r)\,dr \\ 
&=
\frac{c_{k,n}}{2k} \cdot \Gamma\!\left(\frac{k}{2}+1\right) |x|^{1-n} \lambda(|x|),
\end{align*}

where $\lambda$ is defined by \eqref{eq:260129-1},
and $c_{k,n}$ is the constant as in \eqref{eq:psi-w-relation}. Using
\[
\Gamma\!\left(\frac{k}{2}+1\right)= \frac{k}{2}\Gamma\!\left(\frac{k}{2}\right)
\quad \text{and} \quad
\sigma_{k-1}=\frac{2\pi^{k/2}}{\Gamma\!\left(\frac{k}{2}\right)},
\]
we arrive at
\[
\tilde\psi(x)
=
c_1
|x|^{1-n}
\lambda(|x|),
\]
which yields \eqref{eq:260129-11}.

Finally,
since $\tilde\psi$ admits an integrable radially decreasing majorant,
the family $\{\tilde\psi_\varepsilon\}_{\varepsilon>0}$
forms an approximate identity.
Therefore,
\[
\lim_{\varepsilon\downarrow0}
\int_{\varepsilon}^{\infty}
W_t^*\hat f(x)\,\frac{dt}{t}
=
c \,f(x)
\]
both in the topology of $X$ and almost everywhere.
Here
\begin{equation}
c 
=\int_{{\mathbb R}^n}\tilde \psi(|x|)\,dx
=c_1\sigma_{n-1}\int_0^\infty \lambda(s)\,ds =
\frac{\pi^{\frac{n}{2}}}{\Gamma\left(\frac{n-k}{2}\right)}\int_0^\infty \lambda(s)\,ds.
\label{eq 20260308-1}
\end{equation}
This completes the proof of Theorem~\ref{thm:wavelet-reconstruction1}.

We note that the constant $c$ in \eqref{eq 20260308-1} can be further calculated as follows: 
Since
\[
\lambda(s)
=\frac{1}{s\Gamma(\frac12k+1)}\int_0^{s^2}(s^2-r)^{\frac{k}{2}}\tilde w(r)\,dr
\]
we have
\begin{align*}
\int_0^{\infty}  \lambda(s) \, ds 
&=   \int_0^{\infty}   \frac{1}{s} \left( \frac{1}{\Gamma(\frac12k+1)} \int_0^{s^2}(s^2-r)^{\frac{k}{2}}\tilde w(r)\,dr \right) ds \notag \\ 
&= \int_0^{\infty}   I_+^{\frac{1}{2}k+1}[\widetilde{w}](s^2)  \, \frac{ds}{s} \notag \\ 
&= \frac{1}{2} \int_0^{\infty}   I_+^{\frac{1}{2}k+1}[\widetilde{w}](t)  \, \frac{dt}{t}. 
\end{align*}
Here,
\[
\tilde w(s)=s^{\frac{n-k}{2}-1}w(\sqrt{s})
\]
and $w(r)$ is admissible, so it is radial and satisfies 
\eqref{eq:radial-decay-condition} and \eqref{eq:radial-moment-conditions}. 
That is, for some $\beta>k$,
\begin{equation}
\int_1^\infty r^{\beta+n-k-1}|w(r)|\,dr
< \infty,
\notag 
\end{equation}
and also
\begin{equation}
\int_0^\infty r^{j+n-k-1}w(r)\,dr
= 0,
\qquad
j = 0,2,4,\dots,2\left[\frac{k}{2}\right]
\notag 
\end{equation}
hold. Therefore, by the change of variables $r=\sqrt{s}$,
\begin{align}
\int_1^\infty r^{\beta+n-k-1}|w(r)|\,dr 
&= 
\frac{1}{2} \int_1^\infty s^{\frac{\beta+n-k}{2}-1}|w(\sqrt{s})|\,ds = 
\frac{1}{2} \int_1^\infty s^{\frac{\beta}{2}} |\widetilde{w}(s)|\,ds <\infty
\label{eq 20260307-1-noi}
\end{align}
where we note that $\frac{\beta}{2}>\frac{k}{2}$. Similarly,
\begin{align}
\int_{0}^{\infty} r^{j+n-k-1} w(r)\,dr
&=
\frac{1}{2}\int_{0}^{\infty} s^{\frac{j+n-k}{2}-1} w(\sqrt{s})\,ds 
=
\frac{1}{2}\int_{0}^{\infty} s^{J}\,\widetilde{w}(s)\,ds
=0 ,
\label{eq:20260307-2}
\end{align}
where $J=\frac{j}{2}$. The equality \eqref{eq:20260307-2} holds for
\[
J=0,1,2,\ldots,\left[ \frac{k}{2}\right].
\]

Therefore, by Lemma~\ref{lem:260307-1}, we obtain
\begin{align*}
\int_{0}^{\infty} \lambda(s)\,ds
&=
\frac{1}{2}\int_{0}^{\infty} I_{+}^{\frac{k}{2}+1}[\widetilde{w}](t)\,\frac{dt}{t} \\
&=
\frac{1}{2}
\begin{cases}
\displaystyle
\Gamma\!\left(-\frac{k}{2}\right)
\int_{0}^{\infty} t^{\frac{k}{2}} \widetilde{w}(t)\,dt,
& \text{if $k$ is odd}, \\[12pt]
\displaystyle
\frac{(-1)^{1+\frac{k}{2}}}{\left(\frac{k}{2}\right)!}
\int_{0}^{\infty} t^{\frac{k}{2}} \widetilde{w}(t)\log t\,dt,
& \text{if $k$ is even},
\end{cases}
\\[6pt]
&=
\frac{1}{2}
\begin{cases}
\displaystyle
\Gamma\!\left(-\frac{k}{2}\right)
\int_{0}^{\infty} t^{\frac{n}{2}-1} w(\sqrt{t})\,dt,
& \text{if $k$ is odd}, \\[12pt]
\displaystyle
\frac{(-1)^{1+\frac{k}{2}}}{\left(\frac{k}{2}\right)!}
\int_{0}^{\infty} t^{\frac{n}{2}-1} w(\sqrt{t}) \log t\,dt,
& \text{if $k$ is even},
\end{cases}
\\[6pt]
&=
\begin{cases}
\displaystyle
\Gamma\!\left(-\frac{k}{2}\right)
\int_{0}^{\infty} u^{n-1} w(u)\,du,
& \text{if $k$ is odd}, \\[12pt]
\displaystyle
\frac{2(-1)^{1+\frac{k}{2}}}{\left(\frac{k}{2}\right)!}
\int_{0}^{\infty} u^{n-1} w(u)\log u\,du,
& \text{if $k$ is even}.
\end{cases}
\end{align*}
Substituting this into $c$, we obtain
\[
c
=\frac{\pi^{\frac{n}{2}}}{\Gamma\left(\frac{n-k}{2}\right)}\int_0^\infty \lambda(s)\,ds 
= 
\frac{\pi^{\frac{n}{2}}}{\Gamma\left(\frac{n-k}{2}\right)}  
\times 
\begin{cases}
\displaystyle
\Gamma\left(-\frac{k}{2}\right)
\int_{0}^{\infty} s^{n-1}\, w(s)\,ds,
& \text{if $k$ is odd}, \\[3ex]
\displaystyle
\frac{2(-1)^{1+\frac{k}{2}}}{ \left( \frac{k}{2} \right)!}
\int_{0}^{\infty} s^{n-1}\, w(s)\,\log s\,ds,
& \text{if $k$ is even}.
\end{cases}
\]

\subsection{Proof of Theorem \ref{Theorem 3.5}}
We now turn to ridgelet transforms.
Under a suitable normalization,
the Calder\'on reproducing formula
follows from the integral representation
\begin{equation}\label{eq:calderon-integral-representation}
f
=
\int_0^{\infty} f * w_t \,\frac{dt}{t},
\end{equation}
provided that $w = u * v$.
Admissible pairs of functions $u$ and $v$ in
the Calder\'on reproducing formula
are determined by their convolution $w = u * v$,
which must satisfy appropriate cancellation conditions
(see \cite[Section~12]{Rubin2015} and \cite{FJW1991}).

The same deconvolution idea applies if we replace
\eqref{eq:calderon-integral-representation}
by the reconstruction formula
\eqref{eq:dt-over-t-reconstruction}.
We recall the notation
\begin{equation}\label{eq:Ut-def}
U_t f(\tau)
=
t^{k-n}
\int_{\mathbb{R}^n}
f(x)\,
u\!\left(\frac{|x-\tau|}{t}\right)
\,dx,
\end{equation}
and
\begin{equation}\label{eq:Vt-star-def}
V_t^{*}\varphi(x)
=
t^{k-n}
\int_{{\mathcal G}_{n,k}}
\varphi(\tau)\,
v\!\left(\frac{|x-\tau|}{t}\right)
\,d\mu_{\mathcal{G}_{n,k}}(\tau),
\qquad t>0.
\end{equation}

As observed in Rubin~\cite[(3.20)]{Rubin2004}, we have
\[
W_t^* \hat{f} = V_t^* U_t f
\]
for all $f \in L^p(\mathbb{R}^n)$.
Theorem \ref{Theorem 3.5} is clear from Theorem \ref{thm:wavelet-reconstruction1}.
\section{Examples of admissible and non-admissible function spaces}
\label{s4}

In this section we examine several concrete classes of function spaces
in connection with the abstract framework developed above.
Our purpose is to clarify the scope and limitations of
Theorems
\ref{thm:wavelet-reconstruction},
\ref{thm:wavelet-reconstruction1},
and
\ref{Theorem 3.5}
by testing their assumptions on familiar Banach lattices.

In Section \ref{subsec:lorentz-spaces}
we first show that the Lorentz spaces $L^{p,q}({\mathbb R}^n)$
with $1<p<\frac{n}{k}$ and $1\le q\le\infty$
fit naturally into our setting, and that all three reconstruction formulas
are valid in this scale.
We then turn to Lebesgue spaces with variable exponents
$L^{p(\cdot)}({\mathbb R}^n)$
in Section \ref{subsec:variable-exponent},
where the validity of the reconstruction theorems is ensured under
standard $\log$-H\"older continuity and integrability assumptions on the exponent.
Finally, we discuss the classical Morrey spaces
${\mathcal M}^{r_0}_r({\mathbb R}^n)$ in Section \ref{subsec:morrey-spaces},
which, despite the fact that singular integral operators are bounded
${\mathcal M}^{r_0}_r({\mathbb R}^n)$ whenever $1<r \le r_0<\infty$,
fail to satisfy the embedding and integrability conditions required
by our theory.
This contrast highlights the structural nature of the hypotheses
imposed on the ambient space $X$ in the reconstruction theorems.

\subsection{Lorentz spaces}
\label{subsec:lorentz-spaces}

Let $1 \le p < \infty$ and $1 \le q \le \infty$.
The Lorentz space $L^{p,q}({\mathbb R}^n)$ consists of all measurable
functions $f \in L^0({\mathbb R}^n)$ for which the quasi-norm
\[
\|f\|_{L^{p,q}}
:=
\left(
\int_0^{\infty}
\bigl(t^{1/p} f^*(t)\bigr)^q
\,\frac{dt}{t}
\right)^{1/q}
\]
is finite, with the usual modification when $q=\infty$.
Here $f^*$ denotes the non-increasing rearrangement of $f$.

Moreover, $L^{p,q}({\mathbb R}^n)$ is a Banach lattice and
$C_{\rm c}({\mathbb R}^n)$ is dense in $L^{p,q}({\mathbb R}^n)$
whenever $1<p<\infty$ and $1 \le q<\infty$.

Lorentz spaces arise by real interpolation:
for suitable $1<p_0<p<p_1<\infty$ and $0<\theta<1$,
\[
L^{p,q}({\mathbb R}^n)
=
(L^{p_0}({\mathbb R}^n), L^{p_1}({\mathbb R}^n))_{\theta,q},
\]
see \cite{BeLo76}
for example.
So,
if $1<p<\infty$, then $L^{p,q}({\mathbb R}^n)$ is a Banach space.
As a consequence, the Hardy--Littlewood maximal operator $M$
is bounded on $L^{p,q}({\mathbb R}^n)$ for all
$1<p<\infty$ and $1 \le q \le \infty$.
See \cite{ArinoMuckenhoupt1990,CarroRaposoSoria2007,Hunt1966}.

Finally, if
\[
1 < p < r<\frac{n}{k}
\quad\text{and}\quad
1 \le q \le \infty,
\]
then the inclusion
\[
L^{p,q}({\mathbb R}^n)
\subset
L^1({\mathbb R}^n)+L^r({\mathbb R}^n)
\]
holds. Therefore, all the structural assumptions imposed on the Banach lattice
$X$ in Theorems \ref{thm:wavelet-reconstruction},
\ref{thm:wavelet-reconstruction1}, and \ref{Theorem 3.5} are satisfied with $X=L^{p,q}({\mathbb R}^n)$.
See \cite{ArinoMuckenhoupt1990,CarroRaposoSoria2007,Hunt1966}.
Consequently, the corresponding wavelet and Calder\'on-type
reconstruction formulas are valid in Lorentz spaces
$L^{p,q}({\mathbb R}^n)$
for $1<p<\frac{n}{k}$ and $1 \le q \le \infty$.

\subsection{Lebesgue spaces with variable exponents}
\label{subsec:variable-exponent}

In this subsection we restrict ourselves to the \emph{non-weighted} setting.
Let $p(\cdot):\mathbb{R}^n\to(0,\infty)$ be a measurable function.
The variable exponent Lebesgue space
$L^{p(\cdot)}({\mathbb R}^n)$
consists of all measurable functions
$f\in L^0({\mathbb R}^n)$
such that
\[
\| f \|_{L^{p(\cdot)}}
:=
\inf\left\{
\lambda>0:\;
\int_{\mathbb{R}^n}
\left(\frac{|f(x)|}{\lambda}\right)^{p(x)}dx \le 1
\right\}
<\infty.
\]
Equipped with this Luxemburg norm,
$L^{p(\cdot)}({\mathbb R}^n)$
is a Banach function space whenever $1<p_- \le p_+<\infty$.

\medskip

We recall standard notation and regularity conditions for variable exponents.

\begin{definition}
Let $r(\cdot)$ be a variable exponent.
\begin{enumerate}
\item
We define
\[
r_-:=\operatorname*{ess\,inf}_{x\in{\mathbb R}^n} r(x),
\qquad
r_+:=\operatorname*{ess\,sup}_{x\in{\mathbb R}^n} r(x).
\]
\item
The class
${\mathcal P}_0={\mathcal P}_0({\mathbb R}^n)$
consists of all measurable exponents
$r(\cdot):{\mathbb R}^n\to(0,\infty)$
such that $0<r_- \le r_+<\infty$.
The subclass
${\mathcal P}={\mathcal P}({\mathbb R}^n)$
contains those exponents in ${\mathcal P}_0$ satisfying $r_->1$.
\end{enumerate}
\end{definition}

\begin{definition}
Let $r(\cdot)\in{\mathcal P}_0$.
\begin{enumerate}
\item
We say that $r(\cdot)$ satisfies the \emph{local $\log$-H\"older continuity condition}
if there exists a constant $c_*>0$ such that
\begin{equation}\label{logHolder}
|r(x)-r(y)|
\le
\frac{c_*}{\log(|x-y|^{-1})},
\qquad
|x-y|\le \tfrac12.
\end{equation}
\item
If \eqref{logHolder} holds only with $y=0$,
then $r(\cdot)$ is said to be $\log$-H\"older continuous at the origin.
We denote by
${\rm LH}_0={\rm LH}_0({\mathbb R}^n)$
the class of such exponents.
\item
We say that $r(\cdot)$ satisfies the \emph{$\log$-H\"older decay condition at infinity}
if there exist constants $c^*>0$ and $r_\infty\in(0,\infty)$ such that
\begin{equation}\label{decay}
|r(x)-r_\infty|
\le
\frac{c^*}{\log(e+|x|)},
\qquad
x\in{\mathbb R}^n.
\end{equation}
The class of all such exponents is denoted by
${\rm LH}_\infty={\rm LH}_\infty({\mathbb R}^n)$.
\item
Finally, we write
${\rm LH}={\rm LH}({\mathbb R}^n)$
for the class of all exponents satisfying both
\eqref{logHolder} and \eqref{decay}.
\end{enumerate}
\end{definition}

\begin{lemma}\label{lem:2.2}
Suppose that the variable exponent $p(\cdot)$ satisfying $\frac{1}{p(\cdot)} \in {\rm LH}$, and
\[
1\le  p_- \le p_+ \le \infty.
\]
Then the following assertions hold.
\begin{enumerate}
\item
For every cube $Q = Q(z,r)$ with $z \in {\mathbb R}^n$ and $r \le 1$, we have
\[
|Q|^{1/p_-(Q)} \lesssim |Q|^{1/p_+(Q)}.
\]
In particular,
\[
|Q|^{1/p_-(Q)}
\sim
|Q|^{1/p_+(Q)}
\sim
|Q|^{1/p(z)}
\sim
\|\chi_Q\|_{L^{p(\cdot)}({\mathbb R}^n)}.
\]

\item
For every cube $Q = Q(z,r)$ with $z \in {\mathbb R}^n$ and $r \ge 1$, we have
\[
\|\chi_Q\|_{L^{p(\cdot)}({\mathbb R}^n)}
\sim
|Q|^{1/p_\infty}.
\]
\end{enumerate}
\end{lemma}

If a variable exponent $p(\cdot)$ satisfies
$\frac{1}{p(\cdot)}\in{\rm LH}$
and
$1<p_- \le p_+\le\infty$,
then the Hardy--Littlewood maximal operator $M$
is bounded on $L^{p(\cdot)}({\mathbb R}^n)$
\cite{Sawano2018}.
Moreover, under the same assumptions,
$L^{p(\cdot)}({\mathbb R}^n)$
is a Banach lattice belonging to the class $\mathcal A$
as in Definition \ref{def:Muckenhoupt}
introduced in Section~\ref{s1},
and satisfies the density condition
$C_{\rm c}^\infty({\mathbb R}^n)\subset L^{p(\cdot)}({\mathbb R}^n)$.
See \cite{Diening2004,Samko1999}.

If
\[
1\le p_- \le p_+ < \frac{n}{k},
\]
then the inclusion
\[
L^{p(\cdot)}({\mathbb R}^n)
\subset
L^1({\mathbb R}^n)+L^u({\mathbb R}^n)
\quad\text{for some } u<\tfrac{n}{k}
\]
holds.
Consequently,
Theorems
\ref{thm:wavelet-reconstruction},
\ref{thm:wavelet-reconstruction1},
and
\ref{Theorem 3.5}
are applicable to the variable exponent Lebesgue spaces
$L^{p(\cdot)}({\mathbb R}^n)$.
\subsection{Morrey spaces and limitations of the reconstruction theorems}
\label{subsec:morrey-spaces}

Let $1 \le p \le p_0 < \infty$.
For a function $f \in L^p_{\mathrm{loc}}({\mathbb R}^n)$,
the (classical) Morrey norm is defined by
\begin{equation}\label{eq:130709-1A}
\| f \|_{{\mathcal M}^{p_0}_p}
:=
\sup_{x \in {\mathbb R}^n,\, R>0}
|B(x,R)|^{\frac{1}{p_0}-\frac{1}{p}}
\left(
\int_{B(x,R)} |f(y)|^{p} \, dy
\right)^{\frac{1}{p}}.
\end{equation}
The Morrey space ${\mathcal M}^{p_0}_p({\mathbb R}^n)$
consists of all functions $f \in L^p_{\mathrm{loc}}({\mathbb R}^n)$
for which $\|f\|_{{\mathcal M}^{p_0}_p} < \infty$.
This definition extends the classical Lebesgue spaces, since
${\mathcal M}^{p_0}_{p_0}({\mathbb R}^n)$ coincides with
$L^{p_0}({\mathbb R}^n)$ with equality of norms.
We also remark that in \cite{Tanaka2015},
Hitoshi Tanaka introduced the weighted Morrey space
${\mathcal M}^{p_0}_p(w)$, consisting of measurable functions $f$
such that
\[
\| f \|_{{\mathcal M}^{p_0}_p(w)}
:=
\| f \cdot w^{\frac1p} \|_{{\mathcal M}^{p_0}_p}
< \infty 
\]
to construct a counterexample mentioned in Section \ref{s1}.

Let $1 < p \le p_0 < \infty$.
Then the Hardy--Littlewood maximal operator $M$
and the Riesz transforms ${\mathcal R}_j$, $j=1,\dots,n$,
are bounded on ${\mathcal M}^{p_0}_p({\mathbb R}^n)$;
see \cite{ChFr87}.
Moreover, for all $1 \le p \le p_0 < \infty$,
the Morrey space ${\mathcal M}^{p_0}_p({\mathbb R}^n)$
belongs to the class $\mathcal A$ of Banach lattices considered above.

However, Morrey spaces fail to satisfy a key structural assumption
required in Theorems
\ref{thm:wavelet-reconstruction},
\ref{thm:wavelet-reconstruction1},
and
\ref{Theorem 3.5}.
Indeed, as shown by the example of the set $F$
constructed in \cite[Example~11]{book},
the space ${\mathcal M}^{p_0}_p({\mathbb R}^n)$
with $1 < p \le p_0 < \infty$
is never continuously embedded into
$L^1({\mathbb R}^n) + L^u({\mathbb R}^n)$
for any $1 \le u < \infty$.
In particular, the inclusion condition
\[
X \subset L^1({\mathbb R}^n) + L^p({\mathbb R}^n)
\quad \text{for some } p < \frac{n}{k},
\]
assumed in Theorem~\ref{thm:wavelet-reconstruction}
(and hence in Theorems~\ref{thm:wavelet-reconstruction1}
and~\ref{Theorem 3.5}),
fails for such Morrey spaces.

As a further manifestation of this obstruction,
the $k$-plane Radon transform of the characteristic function $\chi_F$
associated with the above set $F$
may be infinite on a set of positive measure.
Consequently, none of the reconstruction formulas established in
Theorems
\ref{thm:wavelet-reconstruction},
\ref{thm:wavelet-reconstruction1},
or
\ref{Theorem 3.5}
is applicable to the classical Morrey spaces
${\mathcal M}^{p_0}_p({\mathbb R}^n)$
when $1 < p < p_0 < \infty$.

On the other hand, since $C_{\rm c}({\mathbb R}^n)$
is dense in the closure
$\widetilde{\mathcal M}^{p_0}_p({\mathbb R}^n)$,
a standard density argument shows that the corresponding
wavelet and Calder\'on-type reconstruction formulas
remain valid in the Morrey-type spaces
$\widetilde{\mathcal M}^{p_0}_p({\mathbb R}^n)$.

\section{Remarks on Banach lattices}
\label{s5}

In this section we compare several structural properties of Banach lattices
that are relevant to convolution estimates and to the abstract framework
developed in the previous sections.

A fundamental observation due to Nogayama
shows that translation invariance of a Banach function space
is equivalent to an $L^{1}$-convolution inequality;
see \cite[Theorem~1.1]{Nogayama26} below.
In particular, when one seeks convolution inequalities
with general $L^{1}$-functions,
translation invariance of the underlying space
plays a decisive role.

\begin{proposition}[{\rm \cite[Theorem~1.1]{Nogayama26}}]
\label{thm:Young-BFS}\
\begin{enumerate}
\item[$(1)$]
Let $X$ be a saturated Banach function space.
Assume that $X$ is translation invariant in the sense that
\begin{equation}\label{eq:translation}
\|f(\cdot - z)\|_{X}
\le
A \|f\|_{X},
\qquad
\text{for all } f \in X \text{ and } z \in {\mathbb R}^{n}.
\end{equation}
Then Young's inequality
\[
\|f * g\|_{X}
\le
A \|f\|_{X} \, \|g\|_{L^{1}({\mathbb R}^{n})}
\]
holds for all $f \in X$ and all $g \in L^{1}({\mathbb R}^{n})$.

\item[$(2)$]
Conversely, let $X$ be a ball Banach function space.
If Young's inequality
\[
\|f * g\|_{X}
\le
B \|f\|_{X} \, \|g\|_{L^{1}({\mathbb R}^{n})}
\]
holds for all $f \in X$ and all $g \in L^{1}({\mathbb R}^{n})$,
then $X$ is translation invariant and \eqref{eq:translation} holds
for all $z \in {\mathbb R}^{n}$ with $A=B$.
\end{enumerate}
\end{proposition}

\medskip

We emphasize that the classical Morrey spaces
${\mathcal M}^p_q({\mathbb R}^n)$ with $1 \le q \le p < \infty$
fall within the scope of Proposition~\ref{thm:Young-BFS},
and therefore enjoy translation invariance and the corresponding
$L^1$-convolution inequality.

In contrast, if $p(\cdot)\in{\rm LH}$ is a variable exponent,
then the space $L^{p(\cdot)}({\mathbb R}^n)$
fails, in general, to satisfy the assumptions of
Proposition~\ref{thm:Young-BFS},
and Young's inequality for convolution with arbitrary $L^1$-functions
does not hold in this setting.
Nevertheless, if
\[
1 \le p_- \le p_+ < \infty,
\]
then $L^{p(\cdot)}({\mathbb R}^n)$
belongs to the class $\mathcal A$ of Banach lattices,
as follows from Lemma~\ref{lem:2.2}.

\begin{center}
{\bf Acknowledgment}
\end{center} 

 \ \ 

This work was partly supported by MEXT Promotion of Distinctive Joint Research Center Program JPMXP0723833165 and Osaka Metropolitan University Strategic Research Promotion Project (Development of International Research Hubs).  
This work was also partly supported by Tokyo City University Start-up Support for Interdisciplinary Research. 
 This work was supported by Grant-in-Aid for Scientific Research (C) (19K03546, 23K03156) (Sawano) and 
 (C) (22H05079, 22H05082, 25K14517) (Tanaka), the Japan Society for the Promotion of Science, 
 and Japan Science and Technology Agency (JST), CREST Grant Number JPMJCR2433 (Tanaka).

\end{document}